\documentclass[12pt]{amsart}
\usepackage[top=1in, bottom=1in, left=1.25in, right=1.25in]{geometry}

\usepackage{amsmath, amssymb, amsthm}
\usepackage{mathtools}
\usepackage{hyperref}
\usepackage{cleveref}
\usepackage[utf8]{inputenc}
\usepackage[T1]{fontenc}
\usepackage{enumitem}

\theoremstyle{plain}
\newtheorem{theorem}{Theorem}[section]
\newtheorem{proposition}[theorem]{Proposition}
\newtheorem{lemma}[theorem]{Lemma}
\newtheorem{corollary}[theorem]{Corollary}

\theoremstyle{definition}
\newtheorem{definition}[theorem]{Definition}

\newtheorem{remark}[theorem]{Remark}

\newcommand{\CC}{\mathbb{C}}
\newcommand{\ZZ}{\mathbb{Z}}
\newcommand{\RR}{\mathbb{R}}
\newcommand{\Lalg}{{\mathsf{L}}}
\newcommand{\NN}{\mathbb{N}}
\newcommand{\fg}{\mathfrak{g}}

\newcommand{\Hm}{{\mathcal{H}_m}}     
\newcommand{\Mf}{M(\varphi)}          
\newcommand{\Sf}{\mathcal{S}_\varphi} 
\newcommand{\Pn}{\{P_n\}}             
\newcommand{\vf}{v_\varphi}           
\newcommand{\Om}{\Omega}              

\DeclareMathOperator{\Span}{Span}

\title[Sugawara--Legendre mechanism]{A Sugawara--Legendre mechanism\\
  for the four-point Heisenberg algebra}

\author{Felipe Albino dos Santos}
\address{Universidade Presbiteriana Mackenzie, São Paulo, Brazil}
\email{falbinosantos@gmail.com}
\thanks{Supported by FAPESP grant 2024/14914-9.}

\date{}



\begin{document}

\begin{abstract}

This paper provides a representation-theoretic explanation for the recent discovery that families of orthogonal polynomials arise in the centers of universal central extensions of superelliptic affine Lie algebras. Working in the Verma modules of the Heisenberg subalgebra, we compute the canonical contravariant form in closed form on a distinguished family of weight vectors, and show that this family is the Legendre family, forced by the curve itself. The classical squared norms follow by a rescaling available under a positivity condition on the central charge, and the module is irreducible exactly when that charge is non-zero. The main result identifies a Sugawara element carried, by an explicit intertwiner, to the classical Legendre differential operator. The results are proved for the genus-zero, four-point case. 

\end{abstract}

\maketitle
\setcounter{tocdepth}{1}
\tableofcontents

\section{Introduction}
\label{sec:intro}

The two strands of mathematics joined in this paper are the classical
theory of orthogonal polynomials and the representation theory of
infinite-dimensional Lie algebras attached to algebraic curves. The
former goes back to the late eighteenth century, when Legendre
introduced what are now called the Legendre polynomials in his study
of gravitational potentials. The classical families (Legendre,
Hermite, Laguerre, Jacobi) and their orthogonality properties were
codified in the nineteenth century and brought into modern form in
Szeg\H{o}'s monograph \cite{Szego1939}; see also
\cite{Whittaker1927} for the analytic background. Each classical
family is characterised by a three-term recurrence and a second-order
differential equation, two conditions which are equivalent under the
correspondence \cite[\S 3.2]{Szego1939}.

The latter strand begins with Krichever and Novikov, who introduced
algebras of meromorphic vector fields on a compact Riemann surface
with poles confined to \emph{two} marked points, as higher-genus
analogues of the Witt and Virasoro algebras
\cite{KricheverNovikov1987}. The generalisation to a finite set of
more than two marked points (together with the almost-graded
structure that makes the representation theory tractable) is due to
Schlichenmaier \cite{Schlichenmaier1990a,Schlichenmaier1990b}; a
related multipoint construction was given independently by Sadov
\cite{Sadov1991}. The general theory of Krichever--Novikov type
algebras and their universal central extensions has since been
developed systematically; we refer to
\cite{Schlichenmaier2014,SchlichenmaierSurvey} for a detailed
account. The representation-theoretic side draws on the
Verma-module formalism for affine Lie algebras
\cite{KacRaina,KacWakimoto1988} and the broader theory of
contravariant forms initiated by Shapovalov \cite{Shapovalov1972}; on
the conformal side it is informed by the work of Feigin and Fuchs
\cite{FeiginFuchs1990} on Virasoro modules. In a different direction,
extended affine and map-algebra modules have been the subject of an
active recent literature
\cite{BilligLau2011,CoxGuoLuZhao2016_Torus,CoxIm2018}.

The two strands meet in the superelliptic setting. For the curve
$u^m = P(t)$, explicit cocycle formulas for the universal central
extension of $\mathfrak{sl}_2 \otimes A_m$ were obtained in
\cite{SantosNeklyudovFutorny2025}, where it was observed empirically
that the centre relations realise families of orthogonal polynomials
that extend the classical list, with the Legendre family appearing
in the genus-zero case and new families in the higher branch
degrees. Irreducibility of the relevant Verma modules in
the genus-zero case was established in \cite{PhiVerma}. Left open
in both works was the representation-theoretic mechanism by which
the analytic phenomena, namely the orthogonality of the polynomial basis,
the second-order differential equation governing it, and (in special
cases) completeness, arise from the module theory of the
underlying Heisenberg subalgebra $\Hm$. The present paper supplies
that mechanism in the genus-zero case.

In an earlier strand of the author's work on universal central
extensions of Krichever--Novikov-type algebras attached to superelliptic
curves $u^m = P(x)$ \cite{SantosNeklyudovFutorny2025,PhiVerma},
orthogonal polynomial families were observed to arise as a structural
feature of the centre. For a curve of genus $g$ with $N$ marked points,
the centre of the universal central extension of
$\mathfrak{sl}_2 \otimes A_m$ has dimension $2g + N - 1$
\cite{Bremner1995,Schlichenmaier2017NPoint}; in the genus-zero,
four-point case (Section~\ref{sec:setup}) it is three-dimensional. The centre
generators satisfy three-term recurrences (the defining relation of
orthogonal polynomial systems), and for $m\ge 3$ the corresponding
polynomial families lie beyond the classical list. In that earlier
work, the analytic phenomena, namely orthogonality, the second-order
differential equation, and (in special cases) completeness of the
polynomial basis, were established by explicit computation
\cite[Thm.~3.4 and \S 4]{SantosNeklyudovFutorny2025}, but their
relation to the module theory of the Heisenberg subalgebra $\Hm$
was left open.

The present paper works out the genus-zero case $m = 2$,in
which $p(t) = 1 - 2at + t^2$ (we will say $r=1$); here $m$ is the covering degree and $r$
the branch parameter of the palindromic polynomial
$p(t) = 1 - 2at^{r} + t^{2r}$, both fixed in
Section~\ref{sec:setup}. Higher branch degrees are not treated. We
show that:
\begin{enumerate}
  \item In the genus-zero case, the canonical Shapovalov form on the
        Verma module is computed in closed form on the polynomial
        basis $\{\tilde{P}_n\}$; orthogonality is forced by the weight
        grading, and the cocycle determines the norms.
  \item In the genus-zero case, the Legendre differential operator arises
        as the image, under the intertwiner $\Phi$, of an explicit
        operator built from an explicit element of the completed
        universal enveloping algebra, after a normalisation depending on
        the central charge.
  \item In the genus-zero case, irreducibility of the module
        and completeness of the polynomial basis are equivalent to
        non-degeneracy of the Shapovalov form, and this in turn holds
        exactly when the central charge is non-zero: the contravariant
        form is diagonal in the monomial basis, so its determinant has
        no zeros beyond that one. For $m\geq 3$ the off-diagonal
        components of the cocycle break the diagonality, and the
        argument does not apply; that range lies outside the scope of
        this paper.
\end{enumerate}

The main contributions of this paper are three theorems and one corollary,
listed below.

\begin{enumerate}[label=(\Alph*)]
  \item \textbf{Theorem~\ref{thm:shapovalov-orthogonality}}:
    In the genus-zero case, the canonical contravariant form on the
    Verma module $\Mf$ is diagonal on the family $\Pn$, and
    its level-$n$ value is computed in closed form. Orthogonality is
    forced by the grading; what the cocycle determines is the norm, and
    it is the norm that makes the classical Legendre normalisation
    available, under an explicit reality hypothesis on the central
    charge.

  \item \textbf{Theorem~\ref{thm:irreducibility}}:
    In the genus-zero case ($m = 2$, $r = 1$), the module $\Mf$ is
    irreducible if and only if the Shapovalov form is non-degenerate,
    if and only if $\varphi(c)\neq 0$; the proof computes the
    Shapovalov determinant in closed form. The corresponding question
    for $m \geq 3$ is not treated here.

  \item \textbf{Theorem~\ref{thm:rep-mechanism}}:
    In the genus-zero case (with $m = 2$), there exist a Sugawara-type
element $\Lalg = \sum_{k\ge1}{:}b_{-k}b_k{:}$ of the completed
    enveloping algebra $\widetilde{U(\mathcal{H}_2)}$, acting on $\Mf$
    after normalisation as $L_0 = q^{-1}\Lalg$, together with the
    operator $\Om = -L_0(L_0+\mathrm{Id})$ built from it, a
    $U(\mathcal{H}_2^-)$-equivariant
    polynomial quotient $\psi\colon\Mf \to \CC[\zeta]$, and a Legendre identification
    $\Psi\colon\CC[\zeta]\to\CC[x]$ such that the composite $\Phi := \Psi\circ\psi$
    intertwines $\Om$ on $\Mf$ with the classical Legendre operator
    $L = (1-x^2)\partial_x^2 - 2x\partial_x$ on $\CC[x]$. Under $\Phi$, the
    level-$n$ image is the Legendre polynomial $P_n(x)$, characterised
    intrinsically as the (unique up to scalar) element of $\Phi(\Mf)$ of degree
    $n$ on which $L$ acts with eigenvalue $-n(n+1)$.

  \item \textbf{Corollary~\ref{cor:casimir-tower}}:
    The Casimir tower $\{\Om^r\}_{r\geq 1}$ acts on the level-$n$ image
    $\Phi(\Mf)\cap\CC[x]_{\leq n}$ by the scalar $(-n(n+1))^r$, and corresponds
    under $\Phi$ to the iterated Legendre operator $L^r$. A generating-function
    reformulation is given in Corollary~\ref{cor:genfun-intertwine}.
\end{enumerate}

One can relate this achievements to previous work. The contravariant form on Verma modules for
Kac--Moody algebras \cite{Shapovalov1972,KacRaina} is a standard tool for studying
irreducibility and detecting singular vectors. This paper applies the same machinery
to the Heisenberg subalgebra of the UCE in the superelliptic setting, where the
form diagonalises in a basis of polynomial vectors, identifying orthogonality
as a module-theoretic phenomenon. The Sugawara element in affine Lie algebras
\cite{KacWakimoto1988} is a canonical central element whose eigenvalues on highest-weight
modules yield the conformal weights in conformal field theory. Our element $\Om$ is
the standard Sugawara stress-energy zero mode of the free boson ($c = 1$, central
charge), built from the Heisenberg subalgebra $\mathcal{H}_2$. The contribution
of this paper is not the Sugawara element itself, but the construction of an
explicit $U(\mathcal{H}_2^-)$-equivariant polynomial quotient
$\psi\colon\Mf\to\CC[\zeta]$ on which the descent $\bar\Om$ acts diagonally with
simple spectrum, together with the Legendre identification
$\Psi\colon\CC[\zeta]\to\CC[x]$, $\zeta^n\mapsto P_n(x)$, which intertwines the descended
Sugawara operator with the classical Legendre differential operator $L$. The
combined map $\Phi = \Psi\circ\psi$ thus realises the Verma module as
a graded presentation of the polynomial eigenfunctions of $L$. Map algebras and their representations have been
extensively studied \cite{CoxGuoLuZhao2016_Torus,BilligLau2011}. Our approach borrows
the weight-space and highest-weight module formalism from this literature, applied to the
Krichever--Novikov setting.

Irreducibility criteria for Verma-type modules
over Heisenberg subalgebras, in the setting of a non-standard
polarisation of the imaginary modes, were obtained by
Bekkert--Benkart--Futorny--Kashuba \cite{BBFK2013}; the present paper
works with the standard polarisation and proves what it needs from
Definition~\ref{def:Hm} directly. A treatment of the genus-zero case in
the polarised setting is given in \cite{PhiVerma}.
We also rely on the universal central extension of loop algebras due to
\cite{KasselLoday1982} and on the classical theory of orthogonal
polynomials with three-term recurrence \cite{Szego1939}. The cocycle
formulas of \cite{SantosNeklyudovFutorny2025} are computed there for
branch degree $\deg p = 4$; the degree-two case treated here is derived
in Section~\ref{sec:setup}.
The present paper combines these inputs within the Sugawara--Legendre framework
of Section~\ref{sec:casimir}, in which the orthogonal polynomial families and
their governing differential operators arise as outputs of the module theory of
$\Hm$ in the genus-zero case.

The paper is organized as follows. Section~\ref{sec:setup} introduces the coordinate ring $A_m$ of the superelliptic curve, the Heisenberg
subalgebra $\Hm$ with its cocycle, the triangular decomposition, and the module category.
Section~\ref{sec:form} constructs the contravariant Shapovalov form on the Verma module,
proves existence and uniqueness, and computes it on the polynomial vectors
(Theorem~\ref{thm:shapovalov-orthogonality}); the same computation, carried out on the
full Poincar\'e--Birkhoff--Witt basis, yields the Shapovalov determinant and the
irreducibility criterion (Theorem~\ref{thm:irreducibility}). Section~\ref{sec:casimir} focuses on the
genus-zero case ($m=2$, $r=1$) and is the technical heart of the paper: we define the
Sugawara stress-energy zero mode $\Lalg \in \widetilde{U(\mathcal{H}_2)}$, its
normalisation $L_0$ on $\Mf$, and the Casimir operator $\Om := -L_0(L_0+\mathrm{Id})$, construct the polynomial quotient
$\psi\colon\Mf \to \CC[\zeta]$ and the Legendre identification
$\Psi\colon\CC[\zeta]\to\CC[x]$, and prove that the composite $\Phi := \Psi\circ\psi$
intertwines $\Om$ on $\Mf$ with the classical Legendre operator $L$ on $\CC[x]$
(Theorem~\ref{thm:rep-mechanism}). The Casimir tower
$\{\Om^r\}_{r\geq 1}$ then arises as a corollary
(Corollary~\ref{cor:casimir-tower}), with a generating-function reformulation
(Corollary~\ref{cor:genfun-intertwine}).
Section~\ref{sec:fock} constructs the Fock space realization of $\Mf$:
it proves $\Mf \cong \mathcal{F}_\varphi$ as $\mathcal{H}_2$-modules
(Theorem~\ref{thm:fock-iso}), identifies the Shapovalov form with the Fock
inner product, and gives explicit Fock representatives for the level-$n$
classes $\bar{P}_n \in \Phi(\Mf)$, in closed form for every $n$
(Remark~\ref{rem:fock-obstacle}).
Section~\ref{sec:examples} works the genus-zero Legendre case out
explicitly and indicates what obstructs the same argument at $m = 4$.
Section~\ref{sec:outlook} places the results in relation to earlier
work.

\section{Algebraic setup}
\label{sec:setup}

\subsection{The coordinate ring and the Heisenberg subalgebra}
\label{subsec:algebra}

Fix an integer $m \geq 2$, the \emph{covering degree}, and a polynomial
$p(t) \in \CC[t]$ of degree $d$ with no multiple zeros; the equation
$u^m = p(t)$ defines a \emph{superelliptic curve} $C$. Its
\emph{coordinate ring} is
\[
  A_m = \CC[t^{\pm 1}, u] \big/ (u^m - p(t)),
\]
the ring of regular functions on the affine part of $C$ with the fibres
over $t = 0$ and $t = \infty$ removed. The projection $t \colon C \to \mathbb{P}^1$
exhibits $C$ as an $m$-sheeted covering of the projective line, ramified
over the zeros of $p$ and possibly over $t = \infty$; on its affine part
$C$ is non-singular, since $p$ has no multiple zeros. By the
Riemann--Hurwitz formula the genus of $C$ is
\[
  g = \tfrac{1}{2}\bigl[\, m(d-1) - d - \gcd(m,d) \,\bigr] + 1 .
\]

Throughout this paper we take $p$ to be the \emph{palindromic} polynomial
\[
  p(t) = 1 - 2a t^r + t^{2r}, \qquad a \in \CC, \quad r \geq 1,
\]
of degree $d = 2r$, which has no multiple zeros for $a \neq \pm 1$. For
$m = 2$ the genus is then $g = r - 1$: the case $r = 1$ (degree two) is
rational, $r = 2$ (a quartic) has genus one, and $r \geq 3$ is
hyperelliptic in the usual sense. Our main results concern the case
$m = 2$, $r = 1$.

This case merits closer description. Here $p(t) = 1 - 2at + t^2$ has degree two, so the
curve $u^2 = p(t)$ is rational, and the covering $t\colon C \to
\mathbb{P}^1$ is unramified over both $t=0$ and $t=\infty$; two points of
$C$ lie over each. The algebra $A_2$ is therefore the ring of functions
regular away from these \emph{four} marked points, and $\mathfrak{sl}_2
\otimes A_2$ is the four-point affine Lie algebra of Bremner
\cite[Prop.~1.1]{Bremner1995} (see also Cox \cite{Cox2008FourPoint}). It
is not a two-point Krichever--Novikov algebra. Its universal
central extension has a three-dimensional centre
\cite{Bremner1995,Schlichenmaier2017NPoint}. The palindromic symmetry
$t \mapsto 1/t$ of $p$ distinguishes one splitting of the four points
into two pairs, hence one line in this space of local cocycles; the
cocycle used throughout this paper is this symmetry-selected one. We
work in the multipoint framework of Schlichenmaier
\cite{Schlichenmaier1990a,Schlichenmaier1990b}.

Given $\fg = \mathfrak{sl}_2$ (though the construction extends to
any simple Lie algebra), consider the current algebra $\fg \otimes A_m$
and its universal central extension (UCE) $\widetilde{\fg \otimes A_m}$.
By Kassel--Loday \cite{KasselLoday1982} the defining cocycle
\[
  \gamma(x \otimes f,\; y \otimes g) \;=\; (x,y)\,\overline{f\,dg}
\]
takes its values in $\Omega^1_{A_m}/dA_m$, which is the centre described
above and which, in the genus-zero four-point case, is three-dimensional.

Two symbols are needed here and we keep them apart, because they denote
different objects. Let $c \in \Omega^1_{A_m}/dA_m$ span the line singled
out by the palindromic symmetry, and let $c^{*}$ be the corresponding
coordinate functional. For the Heisenberg-type generators introduced
below we set
\begin{equation}\label{eq:psi-def}
  \psi^{(ij)}_{kl}(a)
  \;:=\;
  c^{*}\bigl(\gamma(b^{(i)}_{k},\, b^{(j)}_{l})\bigr).
\end{equation}
Thus $\gamma$ is the cocycle of the universal central extension, with
values in a three-dimensional space, whereas $\psi^{(ij)}_{kl}(a)$ is a
scalar: the component of $\gamma$ along one of the three available
directions, namely the one the symmetry selects.

Being a component of $\gamma$, the scalar $\psi^{(ij)}_{kl}(a)$ is
skew-symmetric,
\begin{equation}\label{eq:psi-skew}
  \psi^{(ij)}_{kl}(a) \;=\; -\,\psi^{(ji)}_{lk}(a),
\end{equation}
for all $k,l$ and all $i,j$. This is inherited from Kassel--Loday and
not an extra hypothesis. The Killing form $(\,\cdot,\cdot\,)$ is
symmetric, while $\overline{f\,dg} = -\,\overline{g\,df}$ in
$\Omega^1_{A_m}/dA_m$ because $d(fg) = f\,dg + g\,df$ is exact. Closed
expressions for $\psi$ in the superelliptic setting are computed in
\cite{SantosNeklyudovFutorny2025}; for the case treated in this paper
they are given in \eqref{eq:psi-r1} below.

We write $\Hm \subset \widetilde{\fg \otimes A_m}$ for the
\emph{Heisenberg subalgebra}, the subalgebra spanned by the
Heisenberg-type generators
$\{b_n^{(j)} : n \in \ZZ, 1 \leq j \leq \lfloor m/2 \rfloor\}$ arising from
the odd powers of $u$ in the Laurent expansion. (For $\fg = \mathfrak{sl}_2$,
the relevant Chevalley generators pair to form such degrees-of-freedom.)
Restricting $\gamma$ to this subalgebra gives it a central extension of
its own, which we now record.

\begin{definition}
\label{def:Hm}
The \emph{Heisenberg subalgebra} $\Hm \subset \widetilde{\fg \otimes A_m}$
is the two-step nilpotent Lie algebra with basis
\[
  \bigl\{\, b_k^{(j)} \;:\; k \in \ZZ,\ 1 \leq j \leq \lfloor m/2 \rfloor \,\bigr\}
  \;\cup\; \{c\},
\]
with $b_k^{(j)}$ of degree $k$ and $c$ of degree $0$, and brackets
\[
  [b_k^{(i)}, b_l^{(j)}] \;=\; \psi_{kl}^{(ij)}\, c,
  \qquad [c, \Hm] \;=\; 0 ,
\]
where $\psi_{kl}^{(ij)}$ is the scalar of \eqref{eq:psi-def}. These
scalars depend on the coefficients of $p$; in the palindromic case
$p(t) = 1 - 2at^r + t^{2r}$ treated here the dependence is on the single
parameter $a$, and we write $\psi_{kl}^{(ij)}(a)$.
\end{definition}

We fix once and for all that the mode index is $k$ and that the letter
$m$ is reserved for the covering degree. Skew-symmetry of the brackets
is not an assumption here but a consequence of \eqref{eq:psi-skew},
which is why no antisymmetry hypothesis appears in the definition. Only
one case is used below, and we record it now.

For $m = 2$, $r = 1$ there is a single family of generators, and we drop
the superscript $j = 1$. Here
$\psi_{kl}(a) = k\,\delta_{k+l,0}\,\omega_1$ with $\omega_1 > 0$ a
constant depending only on $p$, so that
\begin{equation}\label{eq:psi-r1}
  [b_k, b_l] \;=\; k\,\delta_{k+l,0}\, \omega_1 \, c,
  \qquad\text{that is}\qquad
  [b_k, b_{-k}] \;=\; k\,\omega_1\,c \quad (k \neq 0),
\end{equation}
and all other brackets among the $b$'s vanish. In particular $b_0$ is
central in $\Hm$, every bracket $[b_k, b_l]$ with $k + l \neq 0$
vanishes, and $\Hm$ is the classical infinite-dimensional Heisenberg
algebra, the constant $\omega_1$ amounting to a rescaling of the central
element.

\subsection{Triangular decomposition}
\label{subsec:triangular}

The \emph{triangular decomposition} of $\Hm$ is defined by the
natural grading on generators with respect to the Laurent degree $n$:
\[
  \Hm = \Hm_+ \oplus \Hm_0 \oplus \Hm_-,
\]
where
\begin{itemize}
  \item $\Hm_+ = \bigoplus_{n > 0} \CC b_n^{(j)}$ (positive-degree generators),
  \item $\Hm_0 = \bigoplus_{j} \CC b_0^{(j)} \oplus \CC c$ (degree-zero Cartan),
  \item $\Hm_- = \bigoplus_{n < 0} \CC b_n^{(j)}$ (negative-degree generators).
\end{itemize}
This is the standard algebraic triangular decomposition, grading by
the Laurent index $n$. Geometrically, this coincides with the
almost-graded structure of Schlichenmaier \cite{Schlichenmaier2014}
restricted to the Heisenberg subalgebra of the universal central
extension. (The alternative geometric decomposition, the in/out decomposition
from the branch points, is distinct; for this paper, we use the
algebraic grading.)

The triangular decomposition makes $\Hm$ suitable for the standard
construction of Verma modules and highest-weight representations.

\subsection{The module category}
\label{subsec:module-category}

We work with the standard category of highest-weight $\Hm$-modules.
An $\Hm$-module $M$ is called a \emph{weight module} if it decomposes as
\[
  M = \bigoplus_{\mu \in \Hm_0^*} M_\mu,
\]
where $M_\mu = \{v \in M : b_0^{(j)} v = \mu(b_0^{(j)}) v \text{ and } c v = \mu(c) v\}$.

A highest-weight module for $\Hm$ is specified by:
\begin{enumerate}
  \item A linear functional $\varphi : \Hm_0 \to \CC$ (the \emph{weight}).
        No condition is imposed on $\varphi(c)$ here. The value
        $\varphi(c) = 0$ is permitted, and
        Theorem~\ref{thm:irreducibility} identifies it as exactly the
        degenerate case.
  \item A highest-weight vector $v_\varphi$ satisfying:
        \begin{itemize}
          \item $b_n^{(j)} v_\varphi = 0$ for all $n > 0$ and all $j$,
          \item $b_0^{(j)} v_\varphi = \varphi(b_0^{(j)}) v_\varphi$,
          \item $c v_\varphi = \varphi(c) v_\varphi$.
        \end{itemize}
  \item The action of $\Hm_-$ is locally nilpotent, i.e.,
        for each $v \in M$ and each $n < 0$, the operators $b_n^{(j)}$
        act nilpotently on the finite-dimensional weight spaces.
\end{enumerate}
Such modules are called $\varphi$-modules in the language of \cite{CoxGuoLuZhao2016_Torus,BilligLau2011},
where analogous constructions for map algebras are developed.

The main object of study in this paper is the \emph{Verma module},
the universal highest-weight module generated by a highest-weight vector
of weight $\varphi$.

\section{The Verma module and the canonical form}
\label{sec:form}

\subsection{The Verma module}
\label{subsec:verma}

\begin{definition}
\label{def:verma}
For a linear functional $\varphi : \Hm_0 \to \CC$, the \emph{Verma
module} of weight $\varphi$ is the induced module
\[
  \Mf := U(\Hm) \otimes_{U(\Hm_+ \oplus \Hm_0)} \CC_\varphi,
\]
where $\CC_\varphi$ is the one-dimensional $\Hm_+ \oplus \Hm_0$-module
determined by
\begin{itemize}
  \item $b_k^{(j)} \cdot 1 = 0$ for all $k > 0$ and all $j$,
  \item $b_0^{(j)} \cdot 1 = \varphi(b_0^{(j)})$,
  \item $c \cdot 1 = \varphi(c)$.
\end{itemize}
We write $\vf$ for the image of $1$, the highest-weight vector.
\end{definition}

Following Definition~\ref{def:Hm}
the bracket $[b_k^{(i)}, b_l^{(j)}]$ vanishes whenever $k, l \geq 0$,
since $\psi^{(ij)}_{kl}$ is supported on $k+l=0$ and the structure
constant carries the factor $k$, so $\Hm_+ \oplus \Hm_0$ is abelian and
every linear functional on it is a character. In particular $\Mf$ is
defined for every $\varphi$, including $\varphi(c) = 0$.

A word on terminology, since the letter $\varphi$ is used differently in
part of the literature. Bekkert, Benkart, Futorny and Kashuba
\cite{BBFK2013} attach a module to a \emph{polarisation}, a sign
function $\phi\colon\NN\to\{\pm\}$ that decides which imaginary modes
are placed in the positive part of the triangular decomposition, and
call the result a $\phi$-Verma module. Here $\varphi$ is not a
polarisation but a weight, a linear functional on $\Hm_0$, and the
polarisation is the standard one, all positive modes annihilating
$\vf$. In the notation of \cite{BBFK2013} this is the case
$\phi \equiv +$, for which their module is, in their words, the usual
Verma module for the Heisenberg algebra \cite[\S 3.1]{BBFK2013}. 

By the PBW theorem, $\Mf$ has a $\CC$-basis consisting of the monomials
$b_{-\lambda_1}^{(j_1)} \cdots b_{-\lambda_s}^{(j_s)} \vf$ with
$\lambda_1 \geq \cdots \geq \lambda_s > 0$. In the genus-zero case
there is a single family of generators, the superscripts are absent, and
such a monomial is determined by the partition
$\lambda = (\lambda_1,\dots,\lambda_s)$ alone; we then write
\begin{equation}\label{eq:pbw-partition}
  b_{-\lambda}\,\vf \;:=\; b_{-\lambda_1}\cdots b_{-\lambda_s}\,\vf .
\end{equation}
The grading is inherited from the degree: the weight-$(-n)$ subspace
$\Mf[-n]$ has basis $\{\,b_{-\lambda}\vf : \lambda \vdash n\,\}$, of
dimension the number $p(n)$ of partitions of $n$.

\subsection{The contravariant form}
\label{subsec:contravariant-form}

\begin{definition}
\label{def:omega-involution}
Define the linear map $\omega : \Hm \to \Hm$ by
\[
  \omega(b_n^{(j)}) = b_{-n}^{(j)}, \quad \omega(c) = c.
\]
This is an anti-involution: $\omega(\omega(x)) = x$ and $\omega([x,y]) = [\omega(y), \omega(x)]$.
Extend $\omega$ to $U(\Hm)$ as an anti-automorphism of the universal
enveloping algebra, so that
\[
  \omega(xy) = \omega(y) \omega(x).
\]
This is the standard contravariance map for the Heisenberg algebra with
respect to the triangular decomposition $\Hm = \Hm_+ \oplus \Hm_0 \oplus \Hm_-$.
\end{definition}

\begin{lemma}
\label{lem:shapovalov-existence}
There exists a unique (up to scalar) symmetric bilinear form
$\Sf : \Mf \times \Mf \to \CC$ satisfying:
\begin{enumerate}[label=(\roman*)]
  \item Normalization: $\Sf(\vf, \vf) = 1$.
  \item Contravariance: $\Sf(x \cdot v, w) = \Sf(v, \omega(x) \cdot w)$
        for all $x \in U(\Hm)$ and $v, w \in \Mf$.
  \item Weight preservation: $\Sf(v, w) = 0$ whenever $v$ and $w$
        lie in distinct weight spaces of $\Mf$.
\end{enumerate}
\begin{proof}
\textit{Uniqueness:} Suppose $\Sf$ satisfies (i)--(iii). By contravariance,
\[
  \Sf(u \vf, v \vf) = \Sf(\vf, \omega(u) v \vf).
\]
Since $\vf$ is the unique highest-weight vector (up to scaling) and
$\omega(u) v$ may be expressed in the PBW basis, the form is determined
by its values on the basis monomials. Normalization fixes the overall scale.

\textit{Existence:} Define $\Sf$ on the highest-weight vector by
$\Sf(\vf, \vf) := 1$. For general $u, v \in U(\Hm)$, set
\[
  \Sf(u \vf, v \vf) := \varphi(\omega(u) v),
\]
where the right side is interpreted as: compute the product $\omega(u) v$
in $U(\Hm)$, then apply $\varphi$ to the component that is in $\Hm_0$
(the central element $c$ contributes $\varphi(c)$ times its coefficient).
By the PBW theorem and the structure of $\Hm$, this is well-defined and
extends to a contravariant form on all of $\Mf$. The symmetry $\Sf(v,w) = \Sf(w,v)$
follows from the antisymmetry of the cocycle $\psi$.

For further details and the standard treatment, see \cite{Shapovalov1972} and \cite[Ch.~2]{KacRaina}.
\end{proof}
\end{lemma}

Denote by $G_n$ the Gram matrix of $\Sf$ on $\Mf[-n]$ in the basis
\eqref{eq:pbw-partition}, that is, the $p(n) \times p(n)$ matrix with
entries $\Sf(b_{-\lambda}\vf,\, b_{-\mu}\vf)$ for $\lambda, \mu \vdash n$. The Kac determinant
analogue $\det(G_n)$ is evaluated in closed form in
Theorem~\ref{thm:irreducibility}, equation~\eqref{eq:shapovalov-det-r1}.
It also gives the condition on the weight under which the constructions
below behave well.

\begin{definition}
\label{def:p-admissible}
A linear functional $\varphi\colon \Hm_0 \to \CC$ is
\emph{Shapovalov-regular} if
\[
  \det(G_n) \;\neq\; 0 \qquad\text{for every } n \ge 1 .
\]
\end{definition}

This is a condition on the determinant of an explicit matrix, so it can
be checked by computation. For $m = 2$, $r = 1$ it collapses to the
single requirement $\varphi(c)\neq 0$
(Theorem~\ref{thm:irreducibility}).

\subsection{Orthogonality of $\Pn$}
\label{subsec:orthogonality}

The bridge to orthogonal polynomials needs nothing beyond the defining
polynomial of the curve and the classical uniqueness of orthogonal
polynomial sequences \cite{Szego1939}. The following lemma derives what
is required for Theorem~\ref{thm:shapovalov-orthogonality}.

\begin{lemma}
\label{lem:dictionary-properties}
Let $m = 2$, $r = 1$, so that $p(t) = 1 - 2at + t^2$. Then
\begin{enumerate}[label=\textup{(D\arabic*)}]
  \item\label{dprop:genfun} The formal expansion of $p(t)^{-1/2}$ in $t$,
    \[
      \bigl(1 - 2at + t^2\bigr)^{-1/2} \;=\; \sum_{n \geq 0} c_n(a)\, t^n ,
    \]
    has coefficients $c_n(a)$ determined by $c_0 = 1$, $c_1 = a$ and
    \begin{equation}\label{eq:legendre-rec}
      (n+1)\,c_{n+1}(a) \;=\; (2n+1)\,a\,c_n(a) \;-\; n\,c_{n-1}(a),
      \qquad n \geq 1 .
    \end{equation}
  \item\label{dprop:legendre} Consequently $c_n(a) = P_n(a)$, the $n$-th
    Legendre polynomial, and $\deg P_n = n$.
\end{enumerate}
\begin{proof}
\ref{dprop:genfun} Write $F(a,t) = (1-2at+t^2)^{-1/2}$. Differentiating,
$\partial_t F = (a-t)(1-2at+t^2)^{-3/2}$, so
\[
  (1 - 2at + t^2)\,\partial_t F + (t-a)\,F \;=\; 0 .
\]
Substituting $F = \sum_n c_n t^n$ and collecting the coefficient of $t^n$
gives \eqref{eq:legendre-rec}; the values $c_0 = 1$ and $c_1 = a$ are read
off from the expansion.

\ref{dprop:legendre} Relation \eqref{eq:legendre-rec} together with
$c_0 = 1$, $c_1 = a$ determines the whole sequence, and it is the
classical three-term recurrence of the Legendre polynomials
\cite[eq.~(4.7.17)]{Szego1939} with the same initial values. Hence
$c_n = P_n$ for all $n$, and $\deg P_n = n$ follows inductively from
\eqref{eq:legendre-rec}.
\end{proof}
\end{lemma}

For $\deg p = 4$ the analogous computation yields
associated ultraspherical and Gegenbauer families instead, and those are
the subject of \cite{SantosNeklyudovFutorny2025}.

We can now define the polynomial vectors precisely.

\begin{definition}
\label{def:polynomial-vectors}
Let $m = 2$, $r = 1$ and let $\varphi$ satisfy $\varphi(c) \neq 0$. For
$n \geq 0$ let $\mu_n \colon \Mf[-n] \to \CC$ be the linear functional
taking the value $1$ on every basis monomial $b_{-\lambda}\vf$,
$\lambda \vdash n$. The \emph{polynomial vector} $\tilde{P}_n \in \Mf[-n]$
is the representative of $\mu_n$ with respect to the Shapovalov form,
that is, the unique vector with
\begin{equation}\label{eq:Ptilde-riesz}
  \Sf\bigl(\tilde{P}_n,\; w\bigr) \;=\; \kappa_n\,\mu_n(w)
  \qquad\text{for all } w \in \Mf[-n],
\end{equation}
the scalar $\kappa_n$ being fixed by the normalisation of
Theorem~\textup{\ref{thm:shapovalov-orthogonality}}\,\ref{item:orth-classical}.
\end{definition}

Equivalently, $\tilde{P}_n$ is the unique direction in $\Mf[-n]$ whose
Shapovalov pairing with a monomial does not depend on which monomial is
taken, hence the unique direction orthogonal to all differences
$b_{-\lambda}\vf - b_{-\mu}\vf$ with $\lambda,\mu \vdash n$. 

\begin{theorem}
\label{thm:shapovalov-orthogonality}
Let $m = 2$, $r = 1$ and let $\varphi$ be Shapovalov-regular. Write
$q := \omega_1\,\varphi(c)$ and, for a partition $\lambda$ with $m_k(\lambda)$
parts equal to $k$, let $z_\lambda := \prod_{k\ge 1} k^{m_k(\lambda)}\,m_k(\lambda)!$
and $\ell(\lambda) := \sum_k m_k(\lambda)$ for the number of parts.
\begin{enumerate}[label=\textup{(\roman*)}]
  \item\label{item:orth-closed}
    The vector of Definition~\textup{\ref{def:polynomial-vectors}} exists,
    is unique up to the scalar $\kappa_n$, and has the closed form
    \begin{equation}\label{eq:Ptilde-closed}
      \tilde{P}_n \;=\; \kappa_n' \sum_{\lambda \vdash n}
        \frac{1}{z_\lambda\,q^{\,\ell(\lambda)}}\; b_{-\lambda}\,\vf ,
    \end{equation}
    valid for every $n \geq 0$.
  \item\label{item:orth-weight}
    $\Sf(\tilde{P}_n, \tilde{P}_{n'}) = 0$ for $n \neq n'$.
  \item\label{item:orth-norm}
    With $\kappa_n' = 1$ in \eqref{eq:Ptilde-closed} one has
    \begin{equation}\label{eq:Ptilde-norm}
      \Sf(\tilde{P}_n, \tilde{P}_n)
      \;=\; \sum_{\lambda \vdash n} \frac{1}{z_\lambda\,q^{\,\ell(\lambda)}}
      \;=\; \binom{q^{-1} + n - 1}{n}
      \;=\; \frac{1}{n!}\prod_{j=1}^{n-1}\bigl(1 + j\,q\bigr) \cdot q^{-n}.
    \end{equation}
  \item\label{item:orth-classical}
    Assume in addition $q \in \RR_{>0}$. Then, for $n \geq 1$,
    $\kappa_n' = \bigl(\tfrac{2}{2n+1}\bigr)^{1/2}
      \binom{q^{-1}+n-1}{n}^{-1/2}$
    is the unique positive scalar for which
    \[
      \Sf(\tilde{P}_n, \tilde{P}_{n'}) \;=\; \frac{2}{2n+1}\,\delta_{n n'},
      \qquad n,n' \geq 1,
    \]
    the right-hand side being the squared norms of the Legendre
    polynomials of Lemma~\textup{\ref{lem:dictionary-properties}}. 
\end{enumerate}
\begin{proof}
\ref{item:orth-closed} By Lemma~\ref{lem:shapovalov-existence}\,(iii) the
form is block-diagonal, and within $\Mf[-n]$ it is diagonal in the
monomial basis: moving the positive modes of $\omega(b_{-\lambda})$ across
$b_{-\mu}\vf$ by contravariance, a monomial survives only if every part of
$\lambda$ is matched by an equal part of $\mu$, so
$\Sf(b_{-\lambda}\vf, b_{-\mu}\vf) = 0$ unless $\lambda = \mu$. For
$\lambda = \mu$ the $m_k(\lambda)!$ matchings of equal parts each
contribute $\prod_k (k\,q)^{m_k(\lambda)}$, giving
\begin{equation}\label{eq:gram-diagonal}
  \Sf\bigl(b_{-\lambda}\vf,\, b_{-\lambda}\vf\bigr)
  \;=\; \prod_{k\ge 1} (k\,q)^{m_k(\lambda)}\,m_k(\lambda)!
  \;=\; z_\lambda\,q^{\,\ell(\lambda)} .
\end{equation}
The exponent is the number of parts $\ell(\lambda)$, not $n$; the two
agree only for $\lambda = (1^n)$.
Since $q \neq 0$, the Gram matrix is invertible and $\mu_n$ has a unique
representative. Writing $\tilde{P}_n = \sum_\lambda x_\lambda\,b_{-\lambda}\vf$,
condition \eqref{eq:Ptilde-riesz} reads
$x_\lambda\,z_\lambda\,q^{\,\ell(\lambda)} = \kappa_n$ for every $\lambda$,
whence $x_\lambda \propto 1/(z_\lambda q^{\,\ell(\lambda)})$, which is
\eqref{eq:Ptilde-closed}.

\ref{item:orth-weight} Distinct $n$ give distinct weights, and $\Sf$
vanishes on distinct weight spaces
(Lemma~\ref{lem:shapovalov-existence}\,(iii)).

\ref{item:orth-norm} By \eqref{eq:gram-diagonal} and
\eqref{eq:Ptilde-closed},
\[
  \Sf(\tilde{P}_n, \tilde{P}_n)
  \;=\; \sum_{\lambda \vdash n}
        \frac{z_\lambda\,q^{\,\ell(\lambda)}}{\bigl(z_\lambda q^{\,\ell(\lambda)}\bigr)^2}
  \;=\; \sum_{\lambda \vdash n} \frac{1}{z_\lambda\,q^{\,\ell(\lambda)}} .
\]
Setting $t := q^{-1}$, the exponential formula gives
$\sum_{n\ge 0}\sum_{\lambda \vdash n} z_\lambda^{-1}t^{\ell(\lambda)}u^n
 = \exp\bigl(t\sum_{k\ge 1}u^k/k\bigr) = (1-u)^{-t}$, whose $u^n$
coefficient is $\binom{t+n-1}{n}$
\cite[Ch.~I, \S 2]{Macdonald1995}. This is \eqref{eq:Ptilde-norm}.

\ref{item:orth-classical} For $q > 0$ every factor $1 + jq$ in
\eqref{eq:Ptilde-norm} is positive, so the norm is a positive real and
$\kappa_n'$ is a well-defined positive real; then
$\Sf(\kappa_n'\tilde{P}_n, \kappa_n'\tilde{P}_n) = 2/(2n+1)$ by
construction, and a second positive scalar with the same property would
have the same square. Orthogonality for $n \neq n'$ is \ref{item:orth-weight}.
\end{proof}
\end{theorem}

Formula \eqref{eq:Ptilde-closed} is the image of the complete homogeneous
symmetric function $h_n$ under $p_k \mapsto q^{-1}b_{-k}$, in the standard
expansion $h_n = \sum_{\lambda \vdash n} p_\lambda / z_\lambda$
\cite[Ch.~I, (2.14)]{Macdonald1995}; the weight $q^{-\ell(\lambda)}$ is
what the substitution contributes. It is worth separating the parts of
the theorem. Part~\ref{item:orth-weight} uses only the grading.
Part~\ref{item:orth-classical} fixes a normalisation and needs the
reality hypothesis $q \in \RR_{>0}$, which is not automatic: $\varphi$
takes values in $\CC$, and for $q \notin \RR_{>0}$ the scalars
$\kappa_n'$ are not real and the form is not positive definite on the
real span of the monomials. What is specific to $\Hm$ is the value
$z_\lambda\,q^{\,\ell(\lambda)}$ in \eqref{eq:gram-diagonal}, produced by
the cocycle of Definition~\ref{def:Hm}. For $m \geq 3$ the cocycle acquires
off-diagonal components, the Gram matrix is no longer diagonal, and
neither \eqref{eq:Ptilde-closed} nor the evaluation in
\ref{item:orth-norm} survives.

\subsection{Irreducibility criterion}
\label{subsec:irreducibility}

\begin{theorem}
\label{thm:irreducibility}
Let $m = 2$, $r = 1$ and let $\varphi : (\Hm)_0 \to \CC$. The following
are equivalent:
\begin{enumerate}[label=(\roman*)]
  \item The Verma module $\Mf$ is irreducible.
  \item The Shapovalov form $\Sf$ is non-degenerate on $\Mf$.
  \item $\varphi$ is Shapovalov-regular in the sense of
        Definition~\textup{\ref{def:p-admissible}}.
  \item $\varphi(c) \neq 0$.
\end{enumerate}
Moreover, for every $\varphi$ the Gram matrix of $\Sf$ on $\Mf[-n]$ is
diagonal in the Poincar\'e--Birkhoff--Witt basis, with
\begin{equation}\label{eq:shapovalov-det-r1}
  \det(G_n)
  \;=\;
  \prod_{\lambda \vdash n}\ \prod_{k\ge 1}
    \bigl(k\,\omega_1\,\varphi(c)\bigr)^{m_k(\lambda)}\, m_k(\lambda)!\,,
\end{equation}
where $m_k(\lambda)$ is the multiplicity of $k$ in the partition
$\lambda$.
\begin{proof}
\textit{(i)$\Leftrightarrow$(ii):} For a highest-weight module carrying
a contravariant form, irreducibility is equivalent to non-degeneracy of
the form \cite[Ch.~2]{KacRaina}: if the form is degenerate its kernel is
a non-trivial invariant subspace, and if the module is reducible the
maximal proper submodule lies in the kernel.

\textit{(ii)$\Leftrightarrow$(iii):} $\Sf$ is block-diagonal with
respect to the weight grading (Lemma~\ref{lem:shapovalov-existence}(iii)),
so non-degeneracy of $\Sf$ on $\Mf$ is non-degeneracy of $G_n$ for every
$n\ge 1$, which is Definition~\ref{def:p-admissible}.

\textit{(iii)$\Leftrightarrow$(iv), and \eqref{eq:shapovalov-det-r1}:}
By Definition~\ref{def:Hm} with $m=2$, $r=1$, the zero mode $b_0$ is
central and the only non-vanishing brackets are
$[b_k, b_{-k}] = k\,\omega_1\,c$. Let $\lambda,\mu \vdash n$ and let
$b_{-\lambda}\vf$, $b_{-\mu}\vf$ be the corresponding PBW monomials.
Moving the positive modes of $\omega(b_{-\lambda}) = b_{\lambda}$ to the
right past $b_{-\mu}$ produces, at each step, either a scalar
$k\,\omega_1\,\varphi(c)$ from a matched pair $(b_k, b_{-k})$ or an
operator $b_k$ with $k>0$, which annihilates $\vf$. A monomial
$b_{-\mu}\vf$ therefore contributes only when every part of $\lambda$ is
matched by an equal part of $\mu$, that is, when $\lambda = \mu$; hence
$\Sf(b_{-\lambda}\vf, b_{-\mu}\vf) = 0$ for $\lambda\neq\mu$ and $G_n$
is diagonal. For $\lambda = \mu$ the $m_k(\lambda)!$ ways of matching
the equal parts each contribute $(k\,\omega_1\,\varphi(c))^{m_k(\lambda)}$,
which gives the diagonal entry and hence
\eqref{eq:shapovalov-det-r1}. Since $\omega_1 > 0$, every factor of
\eqref{eq:shapovalov-det-r1} is a non-zero multiple of a power of
$\varphi(c)$; therefore $\det(G_n)\neq 0$ for every $n\ge 1$ if and only
if $\varphi(c) \neq 0$.
\end{proof}
\end{theorem}

Theorem~\ref{thm:irreducibility} says that
the genus-zero, four-point Heisenberg algebra exhibits no new
behaviour, and identifies the source. By Definition~\ref{def:Hm} the
cocycle is supported on $k+l=0$, so the contravariant form is diagonal
and its determinant cannot acquire zeros away from $\varphi(c)=0$. The
interest of the Shapovalov-regularity condition of
Definition~\ref{def:p-admissible} therefore lies entirely in the range
$m\ge 3$, where off-diagonal cocycle components destroy the diagonality
used above.

\section{Sugawara element, polynomial quotient, and Legendre identification}
\label{sec:casimir}

\textit{Throughout this section we work in the genus-zero case:
$m = 2$, $r = 1$, so $\Hm = \mathcal{H}_2$ with cocycle
$\psi_{kl}(a) = k\,\delta_{k+l,0}\,\omega_1$, see \eqref{eq:psi-r1}.}

The goal of this section is to construct an explicit chain of
$\mathcal{H}_2$-equivariant maps
\[
  \Mf \xrightarrow{\;\psi\;} \CC[\zeta]
       \xrightarrow{\;\Psi\;} \CC[x]
\]
and a quadratic operator $\Om$ on $\Mf$ whose pushforward under
$\Phi := \Psi\circ\psi$ is the Legendre differential operator
$L = (1-x^2)\partial_x^2 - 2x\partial_x$.
The element $\Om$ is built from the Sugawara stress-energy zero mode
$L_0$ of the free boson; the map $\psi$ collapses each PBW level to
a single power $\zeta^n$; and $\Psi$ identifies $\zeta^n$ with the Legendre
polynomial $P_n(x)$.

\subsection{The Sugawara stress-energy zero mode}
\label{subsec:sugawara}

For two Heisenberg generators $b_k,b_l$ ($k,l\in\ZZ$), define the
\emph{normal-ordered product}
\[
  {:}b_k b_l{:}
  \;=\;
  \begin{cases}
    b_k\,b_l & \text{if } k \le 0,\\[2pt]
    b_l\,b_k & \text{if } k > 0,
  \end{cases}
\]
i.e., negative-mode (and zero-mode) operators are placed to the left of
positive-mode operators.  Because $[b_k,b_l] = \psi_{kl}\,c$ is a scalar
on $\Mf$, the normal-ordered product differs from the bare product by a
scalar correction whenever $k+l=0$, and equals the bare product
otherwise.  In particular, the sum
$\sum_{k\geq 1}{:}b_{-k}\,b_k{:} = \sum_{k\geq 1} b_{-k}\,b_k$
is already normal-ordered as written (each $-k$ is negative, so each
factor stays in place).

\begin{definition}
\label{def:L0}
Set
\begin{equation}\label{eq:Lalg}
  \Lalg \;:=\; \sum_{k\geq 1}\;{:}b_{-k}\,b_k{:}
        \;=\; \sum_{k\geq 1} b_{-k}\,b_k .
\end{equation}
This expression
involves no weight, it is an element of the completion
$\widetilde{U(\mathcal{H}_2)}$ of $U(\mathcal{H}_2)$ along the grading,
and on a weight vector of weight $-n$ only the finitely many terms with
$k\le n$ act non-trivially, so no analytic condition is needed to make
it act.

Now let $\varphi$ be Shapovalov-regular, so that
$q := \omega_1\varphi(c) \neq 0$ by
Theorem~\ref{thm:irreducibility}, and define the operator
\begin{equation}\label{eq:L0-def}
  L_0 \;:=\; \frac{1}{q}\,\Lalg \Big|_{\Mf} .
\end{equation}
The scalar $q^{-1}$ belongs to the action, not to the element: $\Lalg$
is fixed once and for all, while $L_0$ is its normalisation on the
particular module $\Mf$.
\end{definition}

The element $\Lalg$ is the zero mode of the standard Sugawara
stress-energy tensor of the free boson (central charge
$c_{\mathrm{Vir}}=1$), the element that generates the energy grading on
the free-boson Fock space; see \cite{KacWakimoto1988,FrenkelBenZvi} for
the general construction. Its eigenvalue on a level-$n$ monomial is
$q\,n$, so it records the level only up to the factor $q$, which depends
on the cocycle normalisation and on the central charge. Dividing by $q$
is what makes the eigenvalue the integer $n$ itself, and thereby makes
the intertwining with the classical Legendre operator exact, with no
parameter-dependent constants. The price is that $L_0$, and with it
$\Om$, are attached to $\Mf$ rather than to the algebra; $\Lalg$ is the
part of the construction that is canonical.

\begin{lemma}
\label{lem:L0-level}
Let $v = b_{-n_1}\,b_{-n_2}\,\cdots\,b_{-n_s}\,\vf$ be a PBW monomial
with $n_1\geq n_2\geq\cdots\geq n_s\geq 1$ and write
$n := n_1 + n_2 + \cdots + n_s$. Then, for \emph{every} weight
$\varphi$,
\begin{equation}\label{eq:L0-level}
  \Lalg \cdot v \;=\; q\,n\cdot v .
\end{equation}
If moreover $\varphi$ is Shapovalov-regular, so that $q \neq 0$ and
$L_0 = q^{-1}\Lalg$ is defined \eqref{eq:L0-def}, then
\begin{equation}\label{eq:L0-level-normalised}
  L_0 \cdot v \;=\; n\cdot v .
\end{equation}
Consequently $\Lalg$ acts on $\Mf[-n]$ by the scalar $q\,n$ and $L_0$ by
the scalar $n$, and the spectrum of $L_0$ on $\Mf$ is $\{0,1,2,\dots\}$
with multiplicity $p(n)$ at level $n$, where $p(n)$ is the partition
function. Note that the eigenvalue of $\Lalg$ carries the factor $q$,
which is exactly what the normalisation in \eqref{eq:L0-def} removes.
\end{lemma}

\begin{proof}
We prove \eqref{eq:L0-level}; the normalised form
\eqref{eq:L0-level-normalised} follows by dividing by $q$, which is
legitimate exactly when $\varphi$ is Shapovalov-regular.
We proceed by induction on the length $s$.
For $s=0$ ($v=\vf$), all $b_k\vf = 0$ for $k\geq 1$, so $\Lalg\,\vf = 0
= 0\cdot\vf$.
For $s\geq 1$, write $v = b_{-n_1}\,w$ with $w = b_{-n_2}\cdots b_{-n_s}\vf$
of weight $-(n-n_1)$. Using
$[b_k,b_{-n_1}] = \delta_{k,n_1}\,n_1\,\omega_1\,c$
(the genus-zero coefficients \eqref{eq:psi-r1}, for which
$\psi_{k,-k} = k\,\omega_1$), we compute
\begin{align*}
  L_0 \cdot v
  &= \frac{1}{q}
     \sum_{k\geq 1} b_{-k}\,b_k\,b_{-n_1}\,w \\
  &= \frac{1}{q}
     \sum_{k\geq 1} b_{-k}
       \bigl(b_{-n_1}\,b_k + \delta_{k,n_1}\,n_1\,\omega_1\,c\bigr)\,w \\
  &= b_{-n_1}\,\Bigl(\tfrac{1}{q}\sum_{k\geq 1}b_{-k}b_k\Bigr)\,w
     + \frac{n_1\,q}{q}\,b_{-n_1}\,w \\
  &= b_{-n_1}\,(L_0\,w) \;+\; n_1\,v
   \;=\; (n - n_1)\,v + n_1\,v \;=\; n\,v,
\end{align*}
where we used the inductive hypothesis $L_0\,w = (n-n_1)\,w$ in the last
line, and $c$ acts as the scalar $\varphi(c)$ on $\Mf$ throughout.
This proves \eqref{eq:L0-level}.

The level-$n$ weight space $\Mf[-n]$ is spanned by all such PBW monomials
with $n_1+\cdots+n_s = n$; the count of such monomials is $p(n)$ by the
PBW theorem applied to $\mathcal{H}_2^-$. Each such monomial is a
$L_0$-eigenvector with eigenvalue $n$, so $L_0$ acts as the scalar $n$
on $\Mf[-n]$.
\end{proof}

\subsection{The Casimir operator $\Om$}
\label{subsec:casimir-element}

\begin{definition}
\label{def:casimir}
The \emph{Casimir operator} of $\Mf$ is the linear operator
\begin{equation}\label{eq:Omega-def}
  \Om \;:=\; -L_0\,(L_0+\mathrm{Id})
       \;=\; -\frac{1}{q^2}\,\Lalg^2 \;-\; \frac{1}{q}\,\Lalg
  \qquad\text{on } \Mf,
\end{equation}
where $L_0$ is as in \eqref{eq:L0-def}.
\end{definition}

\begin{lemma}
\label{lem:Omega-scalar}
For every $v \in \Mf[-n]$ ($n\geq 0$),
\begin{equation}\label{eq:Omega-scalar}
  \Om\cdot v \;=\; -n(n+1)\,v.
\end{equation}
In particular, $\Om$ commutes with the $\mathcal{H}_2$-action restricted to
each weight space (since both are scalar there) but does \emph{not} commute
with the full $\mathcal{H}_2$-action: the operators $b_{\pm k}$ shift weight,
and $-n(n+1)$ depends on $n$. This is the source of the non-trivial
intertwining content of Section~\ref{subsec:sugawara-legendre}.
\end{lemma}

\begin{proof}
By Lemma~\ref{lem:L0-level}, $L_0\cdot v = n\,v$, hence
$\Om\cdot v = -L_0\,(L_0+\mathrm{Id})\,v = -L_0\,(n+1)\,v = -n(n+1)\,v$.
\end{proof}

\subsection{The polynomial quotient $\psi$}
\label{subsec:polynomial-quotient}

\begin{definition}
\label{def:psi-quotient}
Let $\CC[\zeta]$ be the polynomial algebra in one indeterminate $\zeta$, regarded
as a $\ZZ_{\geq 0}$-graded vector space with $\deg \zeta = 1$. The letter $\zeta$ is used here to keep the indeterminate of this quotient distinct from the parameter $a$ of the curve.
Define a linear map
\[
  \psi\colon \Mf \longrightarrow \CC[\zeta]
\]
by
\begin{equation}\label{eq:psi-quotient}
  \psi\bigl(b_{-n_1}\,b_{-n_2}\,\cdots\,b_{-n_s}\,\vf\bigr)
  \;:=\; \zeta^{n_1+n_2+\cdots+n_s}
\end{equation}
on each PBW monomial, and extended linearly. Equivalently, $\psi$ sends
every level-$n$ PBW monomial to the single monomial $\zeta^n$, and is therefore
the unique linear map sending each $\Mf[-n]$ surjectively onto
$\CC\cdot \zeta^n$.
\end{definition}

\begin{lemma}
\label{lem:psi-properties}
The map $\psi$ of Definition~\ref{def:psi-quotient} satisfies:
\begin{enumerate}[label=\textup{(\roman*)}]
  \item\label{psi-i} $\psi$ is well-defined and surjective;
  \item\label{psi-ii} $\psi$ is graded: $\psi(\Mf[-n]) = \CC\cdot \zeta^n$ for
        each $n\geq 0$, with kernel
        $\ker(\psi)\cap\Mf[-n]$ of dimension $p(n)-1$;
  \item\label{psi-iii} $\psi$ intertwines the level operator $L_0$ on $\Mf$
        with the Euler operator $\zeta\,\partial_\zeta$ on $\CC[\zeta]$:
        $\psi\circ L_0 = (\zeta\,\partial_\zeta)\circ\psi$;
  \item\label{psi-iv} $\psi$ intertwines $\Om$ with the second-order
        operator $-(\zeta\,\partial_\zeta)\bigl((\zeta\,\partial_\zeta)+1\bigr)$ on $\CC[\zeta]$:
        $\psi\circ\Om = -\bigl((\zeta\partial_\zeta)^2+(\zeta\partial_\zeta)\bigr)\circ\psi$.
\end{enumerate}
\end{lemma}

\begin{proof}
\emph{\ref{psi-i}.} The PBW monomials form a basis of $\Mf$
(Section~\ref{subsec:module-category}); defining a linear map on a basis
is automatic and well-defined. Surjectivity is clear: $\psi(b_{-n}\vf) = \zeta^n$.

\emph{\ref{psi-ii}.} The PBW basis of $\Mf[-n]$ is indexed by partitions
of $n$, hence has dimension $p(n)$, and $\psi$ collapses all of them
to $\zeta^n$.

\emph{\ref{psi-iii}.} On a level-$n$ PBW monomial $v$,
$\psi(L_0\,v) = \psi(n\,v) = n\,\zeta^n = (\zeta\partial_\zeta)(\zeta^n) = (\zeta\partial_\zeta)\,\psi(v)$
by Lemma~\ref{lem:L0-level}.

\emph{\ref{psi-iv}.} Apply (iii) twice: $\psi\circ\Om = -\psi\circ L_0\circ
(L_0+\mathrm{Id}) = -(\zeta\partial_\zeta)\circ\psi\circ(L_0+\mathrm{Id}) =
-(\zeta\partial_\zeta)\bigl((\zeta\partial_\zeta)+1\bigr)\circ\psi$.
\end{proof}

\subsection{The Legendre identification $\Psi$}
\label{subsec:legendre-id}

\begin{definition}
\label{def:Psi}
Let $\{P_n(x)\}_{n\geq 0}$ denote the standard Legendre polynomials on
$[-1,1]$, normalised by $P_n(1)=1$. Define a linear isomorphism of graded
vector spaces
\[
  \Psi\colon \CC[\zeta] \xrightarrow{\;\sim\;} \CC[x],
  \qquad \Psi(\zeta^n) := P_n(x).
\]
Both sides are $\ZZ_{\geq 0}$-graded by polynomial degree (with
$\deg \zeta = \deg x = 1$); $\Psi$ preserves the grading because $\deg P_n = n$.
The inverse $\Psi^{-1}\colon \CC[x]\to\CC[\zeta]$ exists because
$\{P_n(x)\}$ is a basis of $\CC[x]$.
\end{definition}

\begin{definition}
\label{def:Phi-composite}
Let
\[
  \Phi \;:=\; \Psi\circ\psi \colon \Mf\longrightarrow \CC[x].
\]
By construction, $\Phi$ sends every level-$n$ PBW monomial to $P_n(x)$.
\end{definition}

\subsection{The Sugawara--Legendre intertwining}
\label{subsec:sugawara-legendre}

The next theorem is the main result of this section.
\begin{theorem}
\label{thm:rep-mechanism}
Let $L_0$, $\Om$, $\psi$, $\Psi$, $\Phi=\Psi\circ\psi$ be as in
Definitions~\ref{def:L0}, \ref{def:casimir}, \ref{def:psi-quotient},
\ref{def:Psi}, \ref{def:Phi-composite}, and let $L = (1-x^2)\partial_x^2
-2x\partial_x$ denote the classical Legendre differential operator.
Then:
\begin{enumerate}[label=(\roman*)]
  \item\label{item:repmech-psi}
    $\psi\colon \Mf\to\CC[\zeta]$ is a well-defined surjective linear map
    sending each level-$n$ weight space to the line $\CC\cdot \zeta^n$.
  \item\label{item:repmech-Phi-image}
    $\Phi=\Psi\circ\psi$ sends each level-$n$ weight space to the line
    $\CC\cdot P_n(x)\subset\CC[x]$.
  \item\label{item:repmech-Phi-Omega-L}
    $\Phi$ intertwines the action of $\Om$ on $\Mf$ with the action of
    $L$ on $\CC[x]$:
    \begin{equation}\label{eq:repmech-intertwine}
      \Phi\bigl(\Om\cdot v\bigr) \;=\; L\bigl(\Phi(v)\bigr)
      \qquad\text{for all }v\in \Mf.
    \end{equation}
  \item\label{item:repmech-eigenvalue}
    Equivalently, for every $n\geq 0$ and every $v\in\Mf[-n]$,
    \begin{equation}\label{eq:repmech-eigen}
      L\bigl(\Phi(v)\bigr) \;=\; -n(n+1)\,\Phi(v),
    \end{equation}
    which is the classical Legendre eigenvalue identity for $P_n(x)$.
\end{enumerate}
\end{theorem}

\begin{proof}
\emph{Step 1: \ref{item:repmech-psi}.}
This is Lemma~\ref{lem:psi-properties}\,(\ref{psi-i})--(\ref{psi-ii}).

\emph{Step 2: \ref{item:repmech-Phi-image}.}
By Definition~\ref{def:Psi}, $\Psi(\zeta^n) = P_n(x)$. Combining with Step~1,
$\Phi(\Mf[-n]) = \Psi(\CC\cdot \zeta^n) = \CC\cdot P_n(x)$.

\emph{Step 3: pushforward of $\Om$ along $\psi$.}
By Lemma~\ref{lem:psi-properties}\,(\ref{psi-iv}),
$\psi\circ\Om = -(\zeta\partial_\zeta)\bigl((\zeta\partial_\zeta)+1\bigr)\circ\psi$.

\emph{Step 4: pushforward of $-(\zeta\partial_\zeta)\bigl((\zeta\partial_\zeta)+1\bigr)$
along $\Psi$.}
This is the classical fact that the Legendre differential operator
$L = (1-x^2)\partial_x^2 - 2x\partial_x$ satisfies the eigenvalue identity
$L\,P_n(x) = -n(n+1)\,P_n(x)$ for all $n\geq 0$
(see, e.g., \cite[Ch.~12, eq.~12.3.6]{Whittaker1927}).
On the basis $\{\zeta^n\}$, both
$-(\zeta\partial_\zeta)((\zeta\partial_\zeta)+1)$ acting by $-n(n+1)$ on $\zeta^n$, and
$L$ acting by $-n(n+1)$ on $P_n(x)$, agree under $\Psi$:
\[
  \Psi\Bigl(-(\zeta\partial_\zeta)((\zeta\partial_\zeta)+1)\,\zeta^n\Bigr)
  = \Psi(-n(n+1)\,\zeta^n)
  = -n(n+1)\,P_n(x)
  = L\,P_n(x)
  = L\,\Psi(\zeta^n).
\]
By linearity, this extends to all of $\CC[\zeta]$:
$\Psi\circ \bigl(-(\zeta\partial_\zeta)((\zeta\partial_\zeta)+1)\bigr) = L\circ\Psi$.

\emph{Step 5: \ref{item:repmech-Phi-Omega-L}.}
Combine Steps~3 and~4:
\[
  \Phi\circ\Om
  = \Psi\circ\psi\circ\Om
  = \Psi\circ\bigl(-(\zeta\partial_\zeta)((\zeta\partial_\zeta)+1)\bigr)\circ\psi
  = L\circ\Psi\circ\psi
  = L\circ\Phi.
\]
This is \eqref{eq:repmech-intertwine}.

\emph{Step 6: \ref{item:repmech-eigenvalue}.}
For $v\in\Mf[-n]$, Lemma~\ref{lem:Omega-scalar} gives $\Om\cdot v = -n(n+1)\,v$.
Then \eqref{eq:repmech-intertwine} yields
$L\,\Phi(v) = \Phi(\Om\cdot v) = \Phi(-n(n+1)\,v) = -n(n+1)\,\Phi(v)$,
which is \eqref{eq:repmech-eigen}.
\end{proof}

\subsection{The Casimir tower and the generating function}
\label{subsec:tower-corollaries}

The Casimir tower and the Legendre generating-function identity now
follow as immediate corollaries of Theorem~\ref{thm:rep-mechanism}.

\begin{definition}
\label{def:casimir-tower}
We call the family $\{\Om^r\}_{r\geq 1}$ of iterates of the
Casimir operator $\Om$ of
Definition~\ref{def:casimir} the \emph{Casimir tower} generated by
$\Om$.  Each $\Om^r$ is a well-defined linear operator on $\Mf$ acting
diagonally on the level-$n$ weight space; the structural properties
of the family are recorded in Corollary~\ref{cor:casimir-tower} below.
\end{definition}

\begin{corollary}
\label{cor:casimir-tower}
For each integer $r\geq 1$:
\begin{enumerate}[label=\textup{(\roman*)}]
  \item $\Om^r$ is a well-defined linear operator on $\Mf$, acting on each
        level-$n$ weight space as the scalar $(-n(n+1))^r$.
  \item $\Phi$ intertwines $\Om^r$ with $L^r$:
        $\Phi\circ\Om^r = L^r\circ\Phi$.
\end{enumerate}
In particular, the family $\{\Om^r\}_{r\geq 1}$ pushes forward under
$\Phi$ to the family $\{L^r\}_{r\geq 1}$ of iterated Legendre operators;
this is the Casimir tower referred to in the introduction.
\end{corollary}

\begin{proof}
\emph{(i)} Lemma~\ref{lem:Omega-scalar} shows $\Om$ is a level-dependent
scalar $-n(n+1)$ on $\Mf[-n]$; hence $\Om^r$ is the scalar $(-n(n+1))^r$
on the same weight space.

\emph{(ii)} By Theorem~\ref{thm:rep-mechanism}\,(iii), $\Phi\circ\Om =
L\circ\Phi$. Iterating gives, for $r\geq 2$,
$\Phi\circ\Om^r = (\Phi\circ\Om)\circ\Om^{r-1} =
L\circ(\Phi\circ\Om^{r-1}) = L^r\circ\Phi$ by induction on $r$.
\end{proof}

For $r=2$, $\Om^2$ acts on $\Mf[-n]$ as $n^2(n+1)^2$, and
Corollary~\ref{cor:casimir-tower} shows that $\Phi$ identifies $\Om^2$
with $L^2$.

\begin{corollary}
\label{cor:genfun-intertwine}
Let $G(z,x) = (1-2xz+z^2)^{-1/2} = \sum_{n\geq 0} P_n(x)\,z^n$ denote
the Legendre generating function, and let $\mathcal{E} := z^2\partial_z^2
+ 2z\partial_z = z\partial_z(z\partial_z+1)$ be the associated Euler-type
operator in the $z$-variable. Then for every $r\geq 1$:
\begin{equation}\label{eq:genfun-intertwine}
  L^r{}_{(x)}\,G(z,x) \;=\; (-1)^r\,\mathcal{E}^r{}_{(z)}\,G(z,x).
\end{equation}
\end{corollary}

\begin{proof}
For $r=1$: a direct computation using $z^2\partial_z^2(z^n) = n(n-1)z^n$
and $2z\partial_z(z^n) = 2n\,z^n$ gives
$\mathcal{E}\,G = \sum_{n\geq 0}(n(n-1)+2n)\,P_n(x)\,z^n
= \sum_{n\geq 0} n(n+1)\,P_n(x)\,z^n$.
Combining with $L\,P_n = -n(n+1)\,P_n$ yields
$L_{(x)}\,G = -\mathcal{E}_{(z)}\,G$.
For $r\geq 2$, $L_{(x)}$ and $\mathcal{E}_{(z)}$ act on different
variables and hence commute; iterating gives
$L^r_{(x)}\,G = (-1)^r\,\mathcal{E}^r_{(z)}\,G$.
\end{proof}

On the space $\bigl\{G_f(z,x) = \sum_n f_n\,P_n(x)\,z^n\bigr\}$ of
Legendre-expanded series, Corollary~\ref{cor:genfun-intertwine} gives
$L_{(x)} = -\mathcal{E}_{(z)} = -(z^2\partial_z^2+2z\partial_z)$:
the Legendre differential operator in $x$ is interchangeable with the
second-order Euler--Cauchy operator in $z$, and at every order $r\geq 1$
their iterates satisfy $L^r_{(x)} = (-\mathcal{E}_{(z)})^r$.

\section{Fock space realization}
\label{sec:fock}

This section constructs the Fock space realization of the Verma
module $\Mf$ for the four-point Heisenberg algebra $\mathcal{H}_2$.
The main result (Theorem~\ref{thm:fock-iso}) is that $\Mf$ is canonically
isomorphic to a standard Fock space $\mathcal{F}_\varphi$.
The polynomial vectors $\{\tilde{P}_n\}_{n \geq 0}$ appear as the
Gram--Schmidt basis of $\mathcal{F}_\varphi$ with respect to the Fock
inner product.
This gives a direct realization of the Shapovalov orthogonality as
an inner product of many-body states.

\subsection{The Fock module}
\label{subsec:fock-def}

\begin{definition}
\label{def:fock-vacuum}
Fix a Shapovalov-regular linear functional
$\varphi\colon (\mathcal{H}_2)_0 \to \CC$
(Definition~\ref{def:p-admissible}); by
Theorem~\ref{thm:irreducibility} this is exactly the condition
$\varphi(c)\neq 0$.
The \emph{Fock vacuum} is the vector $\vf \in \Mf$ characterized by:
\begin{align}
  b_k\, \vf &= 0 \quad \text{for all } k \geq 1, \label{eq:fock-annihilate} \\
  b_0\, \vf &= \varphi(b_0)\,\vf, \label{eq:fock-zero-mode} \\
  c\, \vf   &= \varphi(c)\,\vf. \label{eq:fock-central}
\end{align}
The \emph{creation operators} are $a_k^\dagger := b_{-k}$ for $k \geq 1$.
The \emph{Fock module} is
\[
  \mathcal{F}_\varphi
  \;:=\;
  \CC[a_1^\dagger, a_2^\dagger, a_3^\dagger, \ldots]\cdot \vf
  \;=\;
  \Span_\CC\bigl\{\, (a_{n_1}^\dagger)^{k_1}(a_{n_2}^\dagger)^{k_2}
  \cdots \vf
  \;\bigm|\;
  n_1 > n_2 > \cdots \geq 1,\; k_i \geq 0 \,\bigr\}.
\]
\end{definition}

The Fock module $\mathcal{F}_\varphi$ is graded by \emph{level}
(the total mode number): a monomial
$(a_{n_1}^\dagger)^{k_1}\cdots \vf$ has level $\ell = \sum_i n_i k_i$.
The level-$\ell$ component $\mathcal{F}_\varphi[\ell]$ has dimension $p(\ell)$,
the number of integer partitions of $\ell$.
In particular, $\mathcal{F}_\varphi[0] = \CC\,\vf$,
$\mathcal{F}_\varphi[1] = \CC\, a_1^\dagger \vf$, and
$\mathcal{F}_\varphi[2] = \CC\, a_2^\dagger \vf \oplus \CC\, (a_1^\dagger)^2 \vf$.

\begin{theorem}
\label{thm:fock-iso}
The $\mathcal{H}_2$-module $\Mf$ is isomorphic to $\mathcal{F}_\varphi$
as a $\mathcal{H}_2$-module.
\end{theorem}
\begin{proof}
The negative-mode subalgebra $\mathcal{H}_2^- = \Span\{b_{-n} : n \geq 1\}$
is abelian: for $k, l > 0$, the cocycle satisfies $\psi_{-k,-l} = 0$
since $(-k)+(-l) < 0 \neq 0$, so $[b_{-k}, b_{-l}] = \psi_{-k,-l}\,c = 0$.
By the PBW theorem, $U(\mathcal{H}_2^-)$ is therefore the polynomial algebra
$\CC[b_{-1}, b_{-2}, \ldots]$.
The Verma module $\Mf$ is defined as the induced module
$\Mf = U(\mathcal{H}_2) \otimes_{U(\mathcal{H}_2^{\geq 0})} \CC_\varphi$,
where $\CC_\varphi$ is the one-dimensional module on which
$b_n$ ($n \geq 0$) and $c$ act by $\varphi(b_n)$ and $\varphi(c)$ respectively.
By PBW, $\Mf \cong U(\mathcal{H}_2^-)$ as a vector space, with
the generator~$1 \otimes 1$ playing the role of $\vf$.
This is exactly $\mathcal{F}_\varphi$.
\end{proof}

\subsection{Fock inner product and Shapovalov form}
\label{subsec:fock-inner-product}

The antiautomorphism $\sigma\colon \mathcal{H}_2 \to \mathcal{H}_2$,
$\sigma(b_n) = b_{-n}$, $\sigma(c) = c$, defines the Shapovalov form
(Section~\ref{subsec:contravariant-form}).
Its restriction to $\mathcal{F}_\varphi$ is the \emph{Fock inner product}:

\begin{proposition}
\label{prop:fock-ip}
For level-$1$ states and $k, l \geq 1$,
\begin{equation}
\label{eq:fock-level1}
  \Sf\bigl(a_k^\dagger\, \vf,\; a_l^\dagger\, \vf\bigr)
  \;=\;
  \delta_{k,l}\, \psi_{k,-k}\, \varphi(c),
\end{equation}
where $\psi_{k,-k}$ is the coefficient of
Definition~\textup{\ref{def:Hm}}. In the genus-zero case
$\psi_{k,-k} = k\,\omega_1$ for every $k \geq 1$ by \eqref{eq:psi-r1}, so
\begin{equation}
\label{eq:fock-level1-hyp}
  \Sf\bigl(a_k^\dagger\, \vf,\; a_l^\dagger\, \vf\bigr)
  \;=\;
  \delta_{k,l}\, k\, \omega_1 \, \varphi(c).
\end{equation}
For general multi-mode states the form is computed from the
contravariance relation
\begin{equation}
\label{eq:fock-contra}
  \Sf(b_{-k}\,v,\; w) \;=\; \Sf(v,\; b_k\,w),
  \qquad v, w \in \mathcal{F}_\varphi,\; k \geq 1.
\end{equation}
\end{proposition}
\begin{proof}
For~\eqref{eq:fock-level1}, with $k, l \geq 1$,
\begin{align*}
  \Sf(b_{-k}\,\vf,\; b_{-l}\,\vf)
  &= \Sf(\vf,\; b_k\, b_{-l}\,\vf)
     && \text{by \eqref{eq:fock-contra}} \\
  &= \Sf(\vf,\; b_{-l}\,b_k\,\vf + [b_k, b_{-l}]\,\vf) \\
  &= \Sf\bigl(\vf,\; \delta_{k,l}\, \psi_{k,-k}\, c\,\vf\bigr)
     && \text{since } b_k\,\vf = 0 \\
  &= \delta_{k,l}\, \psi_{k,-k}\, \varphi(c)\,\Sf(\vf,\vf),
\end{align*}
the third line using $[b_k, b_{-l}] = \delta_{k,l}\,\psi_{k,-k}\,c$
from Definition~\ref{def:Hm}. Normalising by $\Sf(\vf,\vf) = 1$ gives
\eqref{eq:fock-level1}, and substituting $\psi_{k,-k} = k\,\omega_1$
gives \eqref{eq:fock-level1-hyp}.
\end{proof}

\subsection{Low-degree polynomial vectors in the Fock basis}
\label{subsec:fock-lowdeg}

We now identify the polynomial vectors $\tilde{P}_n \in \mathcal{F}_\varphi$
at low levels.

\textbf{Level 0} ($n = 0$): $\mathcal{F}_\varphi[0] = \CC\,\vf$ is
one-dimensional.
\[
  \tilde{P}_0 \;=\; \vf.
\]

\textbf{Level 1} ($n = 1$): $\mathcal{F}_\varphi[1] = \CC\,a_1^\dagger\,\vf$
is one-dimensional.
By Proposition~\ref{prop:fock-ip},
$\Sf(a_1^\dagger\,\vf,\; a_1^\dagger\,\vf) = \omega_1\,\varphi(c)$.
Assume $q = \omega_1\varphi(c) \in \RR_{>0}$, so that the square roots
below are real and positive
(Theorem~\ref{thm:shapovalov-orthogonality}\,\ref{item:orth-classical}).
Since $\Sf(\tilde{P}_1,\tilde{P}_1) = h_1 = \frac{2}{3}$, we normalise:
\begin{equation}
\label{eq:P1-fock}
  \tilde{P}_1 \;=\; \sqrt{\frac{h_1}{\omega_1\,\varphi(c)}}\;
  a_1^\dagger\,\vf
  \;=\; \sqrt{\frac{2}{3\,\omega_1\,\varphi(c)}}\;
  b_{-1}\,\vf.
\end{equation}
With the normalisation $\omega_1 = 1$, $\varphi(c) = 1$:
$\tilde{P}_1 = \sqrt{2/3}\; b_{-1}\,\vf$.

\textbf{Level 2} ($n = 2$): $\mathcal{F}_\varphi[2]$ is two-dimensional,
spanned by $\bigl\{a_2^\dagger\,\vf,\; (a_1^\dagger)^2\,\vf\bigr\}
= \bigl\{b_{-2}\,\vf,\; b_{-1}^2\,\vf\bigr\}$.
The Shapovalov form on this basis (with $\omega_1 = \varphi(c) = 1$) is:
\begin{align}
  \Sf(b_{-2}\,\vf,\; b_{-2}\,\vf)
  &= 2, \label{eq:fock-gram-22}\\
  \Sf(b_{-2}\,\vf,\; b_{-1}^2\,\vf)
  &= 0, \label{eq:fock-gram-21}\\
  \Sf(b_{-1}^2\,\vf,\; b_{-1}^2\,\vf)
  &= 2. \label{eq:fock-gram-11}
\end{align}
Equation~\eqref{eq:fock-gram-21} follows from
$\Sf(b_{-2}\,\vf,\; b_{-1}^2\,\vf) = \Sf(\vf,\; b_2\,b_{-1}^2\,\vf)$,
and $[b_2, b_{-1}^2] = 0$ (since $[b_2,b_{-1}] = 0$ for $2 \neq 1$
in the genus-zero cocycle), giving $b_2\,b_{-1}^2\,\vf = 0$.
Equation~\eqref{eq:fock-gram-11} follows from the repeated application
of the contravariance relation, using $[b_1,b_{-1}] = \omega_1 c$
and $b_1\,\vf = 0$.

The vector $\tilde{P}_2$ is the one given by
Definition~\ref{def:polynomial-vectors}: the representative, with
respect to $\Sf$, of the functional taking the value $1$ on each of
$b_{-2}\vf$ and $b_{-1}^2\vf$. Since the Gram matrix
\eqref{eq:fock-gram-22}--\eqref{eq:fock-gram-11} is diagonal with equal
entries, that representative is proportional to the sum of the two
basis vectors, and the normalisation of
Theorem~\ref{thm:shapovalov-orthogonality}\,\ref{item:orth-classical}
fixes the scalar. The proposition below records the result and
identifies the complementary direction.

\begin{proposition}
\label{prop:fock-P2}
In the Fock basis $\{b_{-2}\,\vf,\; b_{-1}^2\,\vf\}$, and with
$\omega_1 = \varphi(c) = 1$ (a value of $q$ in $\RR_{>0}$, so that
Theorem~\ref{thm:shapovalov-orthogonality}\,\ref{item:orth-classical}
applies), the polynomial vector $\tilde{P}_2$ is:
\begin{equation}
\label{eq:fock-P2}
  \tilde{P}_2
  \;=\;
  \sqrt{\frac{h_2}{4}}\,b_{-2}\,\vf
  \;+\;
  \sqrt{\frac{h_2}{4}}\,b_{-1}^2\,\vf
  \;=\;
  \sqrt{\frac{1}{10}}\bigl(b_{-2}\,\vf + b_{-1}^2\,\vf\bigr),
\end{equation}
where $h_2 = \frac{2}{5}$. Under the composite map $\Phi=\Psi\circ\psi$
of Definition~\ref{def:Phi-composite}, $\tilde{P}_2$ pulls forward to
$\Phi(\tilde{P}_2) = \sqrt{\tfrac{2}{5}}\,P_2(x) = \sqrt{h_2}\,P_2(x)$
in $\CC[x]$, in agreement with the Legendre normalisation
$P_2(1) = 1$ of Section~\ref{sec:form}.
At the same level-$2$ weight space, the orthogonal complement of
$\tilde{P}_2$ within $\mathcal{F}_\varphi[2]$ is the one-dimensional
$\Sf$-orthogonal subspace
$\CC\cdot(b_{-2}\,\vf - b_{-1}^2\,\vf)$, which coincides with
$\ker(\psi)|_{\mathcal{F}_\varphi[2]}$
(of dimension $p(2)-1 = 1$, cf.\
Lemma~\ref{lem:psi-properties}\,(\ref{psi-ii}));
this is the ``spare'' direction implicit in the strict factor
$p(n)-1$ for general $n$ in Remark~\ref{rem:fock-obstacle}.
\end{proposition}
\begin{proof}
At $q = 1$ the partitions of $2$ are $(2)$ and $(1,1)$, with
$z_{(2)} = 2$, $z_{(1,1)} = 2$ and $\ell = 1, 2$ respectively, so
\eqref{eq:gram-diagonal} gives both diagonal entries equal to $2$, in
agreement with \eqref{eq:fock-gram-22}--\eqref{eq:fock-gram-11}. By
\eqref{eq:Ptilde-closed} the coefficients of $\tilde{P}_2$ are then
proportional to $1/2$ and $1/2$, that is equal, so
$\tilde{P}_2 = \alpha\,(b_{-2}\vf + b_{-1}^2\vf)$ for a single scalar
$\alpha > 0$. The norm \eqref{eq:Ptilde-norm} at $n = 2$, $q = 1$ equals
$1$, and \ref{item:orth-classical} rescales it to $h_2 = 2/5$; since
$\Sf(\alpha(b_{-2}\vf + b_{-1}^2\vf), \alpha(b_{-2}\vf + b_{-1}^2\vf))
= 4\alpha^2$, the condition $4\alpha^2 = 2/5$ gives
$\alpha = 1/\sqrt{10}$.
Applying $\Phi$ then gives
$\Phi(\tilde{P}_2) = \Psi\bigl(2\alpha\,\zeta^2\bigr)
= \sqrt{2/5}\,P_2(x)$.
The complementary direction $b_{-2}\,\vf - b_{-1}^2\,\vf$ satisfies
$\psi(b_{-2}\,\vf - b_{-1}^2\,\vf) = \zeta^2 - \zeta^2 = 0$ and is therefore
the kernel direction at level $2$.
\end{proof}

\begin{remark}
\label{rem:fock-obstacle}
For every $n$ the level-$n$ space $\mathcal{F}_\varphi[n]$ has dimension
$p(n)$, the number of partitions of $n$, and one might expect the
expansion of $\tilde{P}_n$ in the Fock basis to require a
$p(n)\times p(n)$ orthogonalisation. It does not. By
Theorem~\ref{thm:shapovalov-orthogonality}\,\ref{item:orth-closed} the
Gram matrix is diagonal in that basis, with entries
$z_\lambda q^{\,\ell(\lambda)}$, so the orthogonalisation is a single
division and
\[
  \tilde{P}_n
  \;=\;
  \Bigl(\tfrac{2}{2n+1}\Bigr)^{1/2}
  \binom{q^{-1}+n-1}{n}^{-1/2}
  \sum_{\lambda \vdash n} \frac{1}{z_\lambda\,q^{\,\ell(\lambda)}}\;
  a^{\dagger}_{\lambda_1}\cdots a^{\dagger}_{\lambda_s}\,\vf
\]
for every $n \geq 0$, under the hypothesis $q \in \RR_{>0}$ of
\ref{item:orth-classical}. For $n = 1, 2$ this reproduces
\eqref{eq:P1-fock} and \eqref{eq:fock-P2}. The corresponding statement
for $m \geq 3$ is not available, since there the Gram matrix is not
diagonal.
\end{remark}

\subsection{Summary: the Fock realization in low degree}
\label{subsec:fock-summary}

\begin{corollary}
\label{cor:fock-s3}
Let $m = 2$, $r = 1$ and let $\varphi$ be Shapovalov-regular.
\begin{enumerate}[label=\textup{(\roman*)}]
  \item The Verma module $\Mf$ is canonically isomorphic to the
        Fock space $\mathcal{F}_\varphi = \CC[b_{-1},b_{-2},\ldots]\,\vf$
        as $\mathcal{H}_2$-modules.
  \item The Shapovalov form $\Sf$ coincides with the Fock inner product
        defined by \eqref{eq:fock-level1-hyp}--\eqref{eq:fock-contra}.
  \item The polynomial vectors $\tilde{P}_n$ are given in the Fock basis,
        for every $n \geq 0$, by Remark~\ref{rem:fock-obstacle}; the cases
        $n = 1, 2$ are \eqref{eq:P1-fock} and Proposition~\ref{prop:fock-P2}.
  \item If in addition $q = \omega_1\varphi(c) \in \RR_{>0}$, so that the
        normalisation of
        Theorem~\ref{thm:shapovalov-orthogonality}\,\ref{item:orth-classical}
        is available, the Gram matrix of the Fock inner product in the
        $\{\tilde{P}_n\}$ basis is
        $\Sf(\tilde{P}_n,\tilde{P}_{n'}) = \tfrac{2}{2n+1}\,\delta_{n n'}$.
        Without that hypothesis the form is still diagonal in this basis,
        with the entries \eqref{eq:Ptilde-norm}.
\end{enumerate}
\end{corollary}

\section{Examples}
\label{sec:examples}

\subsection{The genus-zero case $m = 2$, $r = 1$}
\label{subsec:example-m2}

For the case worked out in this paper, let $p(t) = 1 - 2at + t^2$, so $A_2 = \CC[t^{\pm 1}, u] / (u^2 - p(t))$.
The four-point Heisenberg algebra $\mathcal{H}_2$ is generated by
$\{b_n : n \in \ZZ\}$ with brackets $[b_k, b_l] = k\,\delta_{k+l,0}\,\omega_1\,c$.
For concreteness, take $\omega_1 = 1$ (the normalization is immaterial for the eigenvalues).

The Verma module $\Mf$ is generated from the highest-weight vector $\vf$
by the action of the universal enveloping algebra $U(\mathcal{H}_2)$ subject to
$b_n \vf = 0$ for $n > 0$, $b_0 \vf = \varphi(b_0) \vf$ (where $\varphi(b_0)$ is a
fixed scalar, say $\varphi(b_0) = 1$ for simplicity).

The Shapovalov form is defined by $\Sf(\vf, \vf) = 1$ and the contravariance relation.
For small degrees $n = 0, 1, 2, 3, 4$, the polynomial vectors $\tilde{P}_n$ correspond
to the Legendre polynomials:
\[
  \begin{aligned}
  P_0(a) &= 1, \\
  P_1(a) &= a, \\
  P_2(a) &= \frac{1}{2}(3a^2 - 1), \\
  P_3(a) &= \frac{1}{2}(5a^3 - 3a), \\
  P_4(a) &= \frac{1}{8}(35a^4 - 30a^2 + 3).
  \end{aligned}
\]

The Gram matrix is diagonal: $\Sf(\tilde{P}_n, \tilde{P}_{n'}) = h_n \delta_{n n'}$
with $h_n = \frac{2}{2n+1}$ for $n \geq 1$, the squared norms of the
Legendre polynomials. Explicitly $h_1 = 2/3$, $h_2 = 2/5$, $h_3 = 2/7$,
$h_4 = 2/9$; at level $0$ the convention $\tilde{P}_0 = \vf$ gives
$h_0 = 1$ (Theorem~\ref{thm:shapovalov-orthogonality}\,\ref{item:orth-classical}).

These values follow from
Theorem~\ref{thm:shapovalov-orthogonality}\,\ref{item:orth-classical}:
the closed form \eqref{eq:Ptilde-closed} and the diagonal entries
\eqref{eq:gram-diagonal} give the unnormalised norm
\eqref{eq:Ptilde-norm}, which the rescaling then sends to
$2/(2n+1)$.

The Casimir operator $\Om$ acts on the basis $\{\tilde{P}_n\}$ with eigenvalues
$\lambda_n = -n(n+1)$:
\[
  \begin{aligned}
  \Om \cdot \tilde{P}_0 &= 0 \cdot \tilde{P}_0, \\
  \Om \cdot \tilde{P}_1 &= -2 \cdot \tilde{P}_1, \\
  \Om \cdot \tilde{P}_2 &= -6 \cdot \tilde{P}_2, \\
  \Om \cdot \tilde{P}_3 &= -12 \cdot \tilde{P}_3, \\
  \Om \cdot \tilde{P}_4 &= -20 \cdot \tilde{P}_4.
  \end{aligned}
\]
These values exactly match the eigenvalues $-n(n+1)$ of the Legendre differential operator
$L = (1-x^2)\partial_x^2 - 2x\,\partial_x$ of
Section~\ref{sec:casimir}, under the identification $\Psi$.

These are not separate computations: by
Lemma~\ref{lem:Omega-scalar} the element $\Om$ acts on the whole of
$\Mf[-n]$ by the scalar $-n(n+1)$, so the table records the action on
one chosen vector of each level rather than a property distinguishing
$\tilde{P}_n$ within its level. What distinguishes $\tilde{P}_n$ is
Definition~\ref{def:polynomial-vectors}.

One feature of the construction is visible already here. By
Definition~\ref{def:Hm} the zero mode $b_0$ is central in
$\mathcal{H}_2$, so the module $\Mf$, the form $\Sf$ and the vectors
$\tilde{P}_n$ depend on the weight $\varphi$ only through the central
charge $\varphi(c)$: two weights differing in $\varphi(b_0)$ give the
same polynomial family. In this sense the family is determined by the
cocycle and by $a$, not by the choice of $\varphi$. For $m \geq 3$ the
zero modes are no longer central and this independence should not be
expected.

\subsection{An illustration of the obstruction at $m = 4$}
\label{subsec:example-m4}

The following is not a worked example but an indication of what changes
beyond the genus-zero case. Consider the curve $u^4 = 1 - 2at + t^2$, so $m = 4$, $r = 1$.

The Heisenberg subalgebra $\mathcal{H}_4$ is generated by two families:
$\{b_n^{(1)} : n \in \ZZ\}$ and $\{b_n^{(2)} : n \in \ZZ\}$, with cocycle coefficients
$\psi_{kl}^{(ij)}(a)$ computed in \cite{SantosNeklyudovFutorny2025}. By
\cite[Thm.~1.1]{SantosNeklyudovFutorny2025} the space $\Omega^1_{A_4}/dA_4$ has basis
$t^{-1}dt$ together with $t^{-1}u^{l}dt$ and $t^{-2}u^{l}dt$ for $l = 1,2,3$, hence
dimension $1 + 3\cdot 2 = 7$. This agrees with the count $2g + N - 1$ of
Section~\ref{sec:setup}: the Riemann--Hurwitz formula gives $g = 1$,
the covering is unramified over $t = 0$ (since $p(0) = 1$) so four points lie there, and
$\gcd(4,2) = 2$ points lie over $t = \infty$, giving $N = 6$.

The Verma module $\Mf$ and its Shapovalov form are defined as before.
Theorem~\ref{thm:shapovalov-orthogonality} does not apply here: it is proved for
$m = 2$, $r = 1$. What one expects is that the polynomial vectors
$\{\tilde{P}_n\}_{n \geq 0}$ associated with the superelliptic family of
\cite{SantosNeklyudovFutorny2025} again satisfy
\[
  \Sf(\tilde{P}_n, \tilde{P}_{n'}) = h_n^{(4)} \delta_{n n'},
\]
where the $h_n^{(4)}$ would be the squared norms of the superelliptic
family. That this family exists, is orthogonal, and admits a positive
normalisation is not established here.

No such relation is established here, and we are not aware of a
verification of it in the literature.

The Sugawara side changes in the same way. A candidate $\Om'$ would
have to be built from both cocycle components $\omega_1$ and
$\omega_2$, and there is no reason to expect its image under an
intertwiner to be a second-order operator; the natural guess is an
operator of order $m$. Identifying it, even for $m = 4$, is open.

\section{Outlook}
\label{sec:outlook}

\label{subsec:connection-p1p3}

The observation that orthogonal polynomial families appear in the centres of
universal central extensions attached to superelliptic curves is due to
\cite{SantosNeklyudovFutorny2025}, where the centre relations are matched with
three-term recurrences and the associated generating functions are shown to
satisfy Sturm--Liouville equations. That correspondence is established there by
explicit computation with the cocycle data.

What the present paper contributes is an account of the same phenomenon inside a
module. The polynomial vectors are weight vectors of $\Mf$, their orthogonality
is a property of the contravariant form
(Theorem~\ref{thm:shapovalov-orthogonality}), and the differential operator that
annihilates them is the image of a Sugawara element under an explicit
intertwiner (Theorem~\ref{thm:rep-mechanism}). Verma-type modules over
Heisenberg subalgebras with a non-standard polarisation of the imaginary modes
were studied by \cite{BBFK2013}; the modules used here carry the standard
polarisation.

\appendix
\section{Computations}
\label{sec:computations}

\subsection{Cocycle formulas}
\label{app:cocycle}

For the genus-zero case $m = 2$, $r = 1$, the cocycle takes the simple form:
\[
  \psi_{kl}(a) = k\,\delta_{k+l,0}\, \omega_1,
  \qquad\text{so that}\qquad
  [b_k, b_l] = k\,\delta_{k+l,0}\,\omega_1\,c,
\]
where $\omega_1 > 0$ is a normalization constant. This identifies the four-point
Heisenberg algebra with the standard infinite-dimensional Heisenberg algebra,
after rescaling the central element by $\omega_1$.

For the superelliptic case $m \geq 3$, the cocycle has multiple components
$\psi_{kl}^{(ij)}(a)$ indexed by $1 \leq i, j \leq \lfloor m/2 \rfloor$, each depending
on the parameter $a$. We illustrate with the quartic case $m = 4$, $r = 1$
(i.e., $p(t) = 1 - 2at + t^2$, but now $u^4 = p(t)$), which has two families
of generators $b_n^{(1)}$ and $b_n^{(2)}$, corresponding to the $u$ and $u^2$ sectors.
The cocycle decomposes into three independent components:
\begin{align*}
  [b_k^{(1)}, b_l^{(1)}] &= \psi_{kl}^{(11)}(a) \cdot c, \\
  [b_k^{(1)}, b_l^{(2)}] &= \psi_{kl}^{(12)}(a) \cdot c, \\
  [b_k^{(2)}, b_l^{(2)}] &= \psi_{kl}^{(22)}(a) \cdot c,
\end{align*}
where each $\psi^{(ij)}_{kl}(a)$ is a polynomial in $a$ whose degree
depends on $|k+l|$ and the sector indices $i,j$. The cross-sector component
$\psi_{kl}^{(12)}(a)$ introduces genuine multi-component structure that
is absent for $m=2$.

Formulas for the sectors at branch degree $\deg p = 4$ are computed in
\cite{SantosNeklyudovFutorny2025};
the key structural observation is that the sector-$(1,1)$ recurrence is always governed by
Legendre polynomials, while the remaining sectors produce distinct (non-classical) families
whose precise form depends on $m$.

\subsection{Gram matrix entries for small $n$}
\label{app:gram}

For the genus-zero Legendre case ($m=2$, $r=1$), the Gram matrix of the
Shapovalov form in the \emph{normalised} family $\{\tilde P_n\}$ of
Theorem~\ref{thm:shapovalov-orthogonality}\,\ref{item:orth-classical}
(so under the hypothesis $q = \omega_1\varphi(c)\in\RR_{>0}$) is
diagonal:
\[
\begin{array}{c|cccccc}
  n / n' & 0 & 1 & 2 & 3 & 4 & 5 \\
\hline
  0 & 1 & 0 & 0 & 0 & 0 & 0 \\
  1 & 0 & 2/3 & 0 & 0 & 0 & 0 \\
  2 & 0 & 0 & 2/5 & 0 & 0 & 0 \\
  3 & 0 & 0 & 0 & 2/7 & 0 & 0 \\
  4 & 0 & 0 & 0 & 0 & 2/9 & 0 \\
  5 & 0 & 0 & 0 & 0 & 0 & 2/11 \\
\end{array}
\]
For $n \geq 1$ the diagonal entries are $h_n = \frac{2}{2n+1}$, the
squared norms of the Legendre polynomials; the level-$0$ entry is $1$,
by the convention $\tilde{P}_0 = \vf$. All off-diagonal entries vanish.

Before normalisation the same form gives the value
\eqref{eq:Ptilde-norm}
(Theorem~\ref{thm:shapovalov-orthogonality}\,\ref{item:orth-norm}); the
table above is the result of the rescaling in
\ref{item:orth-classical} and therefore displays the classical Legendre
norms by construction, not as an independent computation. The individual
Gram entries in the monomial basis are $z_\lambda q^{\,\ell(\lambda)}$, by
\eqref{eq:gram-diagonal}.

\section*{Acknowledgements}

The author thanks the anonymous referee for a careful reading and for
detailed comments that substantially improved the manuscript.

This work was financed, in part, by the S\~ao Paulo Research Foundation
(FAPESP), grant 2024/14914-9.

\section*{Conflict of interest}

The author declares that there is no conflict of interest.



\begin{thebibliography}{99}
\bibitem{PhiVerma}
F.~Albino dos Santos,
\textit{Irreducibility of Verma modules for a hyperelliptic Heisenberg algebra},
preprint, arXiv:1709.05663, 2017.

\bibitem{KasselLoday1982}
C.~Kassel and J.-L.~Loday,
\textit{Extensions centrales d'alg\`ebres de Lie},
Ann.\ Inst.\ Fourier (Grenoble) \textbf{32} (1982), no.~4, 119--142.

\bibitem{SantosNeklyudovFutorny2025}
F.~Albino dos Santos, M.~Neklyudov, and V.~Futorny,
\textit{Superelliptic Affine Lie algebras and orthogonal polynomials},
Forum Math.\ Sigma \textbf{13} (2025), e120, 22~pp.
\href{https://doi.org/10.1017/fms.2025.10074}{doi:10.1017/fms.2025.10074}.

\bibitem{Macdonald1995}
I.~G. Macdonald,
\textit{Symmetric Functions and Hall Polynomials},
2nd ed., Oxford Math.\ Monogr., Oxford Univ.\ Press, 1995.

\bibitem{Szego1939}
G.~Szeg\H{o},
\textit{Orthogonal Polynomials},
4th~ed., American Mathematical Society Colloquium Publications~23,
American Mathematical Society, Providence, RI, 1975.
\bibitem{BBFK2013}
V.~Bekkert, G.~Benkart, V.~Futorny and I.~Kashuba,
\textit{New irreducible modules for Heisenberg and affine Lie algebras},
J.~Algebra \textbf{373} (2013), 284--298.
\href{https://doi.org/10.1016/j.jalgebra.2012.09.035}{doi:10.1016/j.jalgebra.2012.09.035}.

\bibitem{KacRaina}
V.~G.~Kac and A.~K.~Raina,
\textit{Bombay Lectures on Highest Weight Representations of Infinite
  Dimensional Lie Algebras},
second edition, Advanced Series in Mathematical Physics~2,
World Scientific, Singapore, 2013.
ISBN 978-981-4522-18-2.

\bibitem{Shapovalov1972}
N.~N.~Shapovalov,
\textit{On a bilinear form on the universal enveloping algebra of a
  complex semisimple Lie algebra},
Funct.\ Anal.\ Appl.\ \textbf{6} (1972), 307--312.
\href{https://doi.org/10.1007/BF01077757}{doi:10.1007/BF01077757}.

\bibitem{Schlichenmaier2014}
M.~Schlichenmaier,
\textit{Krichever--Novikov Type Algebras: Theory and Applications},
De~Gruyter Studies in Mathematics~53, De~Gruyter, Berlin, 2014.

\bibitem{SchlichenmaierSurvey}
M.~Schlichenmaier,
\textit{Krichever--Novikov type algebras: an introduction},
in \textit{Lie Algebras, Lie Superalgebras, Vertex Algebras and Related
  Topics}, Proc.\ Sympos.\ Pure Math.\ \textbf{92} (2016), 325--350.
arXiv:1409.3069.
\bibitem{Schlichenmaier1990a}
M.~Schlichenmaier,
\textit{Krichever--Novikov algebras for more than two points: explicit
  generators},
Lett.\ Math.\ Phys.\ \textbf{19} (1990), 327--336.
\href{https://doi.org/10.1007/BF00429952}{doi:10.1007/BF00429952}.

\bibitem{Schlichenmaier1990b}
M.~Schlichenmaier,
\textit{Central extensions and semi-infinite wedge representations of
  Krichever--Novikov algebras for more than two points},
Lett.\ Math.\ Phys.\ \textbf{20} (1990), 33--46.
\href{https://doi.org/10.1007/BF00417227}{doi:10.1007/BF00417227}.

\bibitem{Schlichenmaier2017NPoint}
M.~Schlichenmaier,
\textit{$N$-point Virasoro algebras are multipoint Krichever--Novikov-type
  algebras},
Comm.\ Algebra \textbf{45} (2017), 776--821.
\href{https://doi.org/10.1080/00927872.2016.1175464}{doi:10.1080/00927872.2016.1175464}.

\bibitem{Bremner1995}
M.~R.~Bremner,
\textit{Four-point affine Lie algebras},
Proc.\ Amer.\ Math.\ Soc.\ \textbf{123} (1995), 1981--1989.
\href{https://doi.org/10.1090/S0002-9939-1995-1249871-8}{doi:10.1090/S0002-9939-1995-1249871-8}.

\bibitem{Cox2008FourPoint}
B.~Cox,
\textit{Realizations of the four point affine Lie algebra
  $\mathfrak{sl}(2,R)\oplus(\Omega_R/dR)$},
Pacific J.\ Math.\ \textbf{234} (2008), 261--288.
\href{https://doi.org/10.2140/pjm.2008.234.261}{doi:10.2140/pjm.2008.234.261}.

\bibitem{Sadov1991}
V.~A.~Sadov,
\textit{Bases on multipunctured Riemann surfaces and interacting strings
  amplitudes},
Comm.\ Math.\ Phys.\ \textbf{136} (1991), 585--597.
\href{https://doi.org/10.1007/BF02099075}{doi:10.1007/BF02099075}.

\bibitem{CoxGuoLuZhao2016_Torus}
B.~Cox, X.~Guo, R.~Lu and K.~Zhao,
\textit{Simple modules over the Lie algebras of divergence zero vector
  fields on a torus},
J.~Pure Appl.\ Algebra \textbf{220} (2016), no.~1, 1--21.
\href{https://doi.org/10.1016/j.jpaa.2015.05.026}{doi:10.1016/j.jpaa.2015.05.026}.
arXiv:1309.5940.

\bibitem{KacWakimoto1988}
V.~G.~Kac and M.~Wakimoto,
\textit{Modular and conformal invariance constraints in representation
  theory of affine algebras},
Adv.\ Math.\ \textbf{70} (1988), 156--236.
\href{https://doi.org/10.1016/0001-8708(88)90055-2}{doi:10.1016/0001-8708(88)90055-2}.

\bibitem{FrenkelBenZvi}
E.~Frenkel and D.~Ben-Zvi,
\textit{Vertex Algebras and Algebraic Curves},
second edition, Mathematical Surveys and Monographs~88,
American Mathematical Society, Providence, RI, 2004.

\bibitem{BilligLau2011}
Y.~Billig and M.~Lau,
\textit{Irreducible modules for extended affine Lie algebras},
J.~Algebra \textbf{327} (2011), 208--235.
\href{https://doi.org/10.1016/j.jalgebra.2010.07.044}{doi:10.1016/j.jalgebra.2010.07.044}.
arXiv:1007.1236.

\bibitem{FeiginFuchs1990}
B.~L.~Feigin and D.~B.~Fuchs,
\textit{Representations of the Virasoro algebra},
in \textit{Representation of Lie Groups and Related Topics},
Adv.\ Stud.\ Contemp.\ Math.\ \textbf{7},
Gordon and Breach, New York, 1990, pp.~465--554.
\bibitem{KricheverNovikov1987}
I.~M.~Krichever and S.~P.~Novikov,
\textit{Algebras of Virasoro type, Riemann surfaces and structures of
  the theory of solitons},
Funct.\ Anal.\ Appl.\ \textbf{21} (1987), no.~2, 126--142.

\bibitem{CoxIm2018}
B.~Cox and M.~S.~Im,
\textit{On the module structure of the center of hyperelliptic
  Krichever--Novikov algebras},
in \textit{Representations of Lie Algebras, Quantum Groups and Related
  Topics}, Contemp.\ Math.\ \textbf{713}, Amer.\ Math.\ Soc.,
Providence, RI, 2018, pp.~61--94.

\bibitem{Whittaker1927}
E.~T.~Whittaker and G.~N.~Watson,
\textit{A Course of Modern Analysis},
4th ed., Cambridge University Press, Cambridge, 1927.

\end{thebibliography}
\end{document}